\documentclass[a4paper,12pt]{amsart} \usepackage{euscript,eufrak,verbatim}
\usepackage[final]{epsfig}
\usepackage{psfrag}
 \def\nn{\nonumber}  \def\a{\alpha} 
\def\e{\varepsilon}
\renewcommand{\liminf}{\mathop{{\underline {\hbox{{\rm lim}}}}}}
\renewcommand{\limsup}{\mathop{{\overline {\hbox{{\rm lim}}}}}}
\DeclareMathOperator{\diam}{diam} 
 \def\R{\mathbb{R}} \def\S{\mathbb{S}} \def\N{\mathbb{N}}

\def\O{\EuScript{O}}
\def\fC{\EuScript{C}}
\def\I{{J}^{\kappa}_n(s)}
\def\In{I^{\kappa}(s)}
\def\Fn{F^{\kappa}(s)}
\def\Ihn{\mathcal{I}^{\kappa}(x)}
\def\Fhn{\EuScript{F}^{\kappa}(x)}
\def\hb{\underline{h}}

\def\1{\mathbf 1}
\def\mm{\mu_{\max}}

\def\G{\mathcal{G}}

\def\bequ{\begin{equation}}
\def\nequ{\end{equation}}

\def\bdef{\begin{defn}}
\def\ndef{\end{defn}}

\def\bthm{\begin{thm}}
\def\nthm{\end{thm}}

\def\bprop{\begin{prop}}
\def\nprop{\end{prop}}

\def\brmk{\begin{remarks}}
\def\nrmk{\end{remarks}}

\def\nn{\mathbb{N}}
\def\rr{\mathbb{R}}

\def\bdes{\begin{description}}
\def\ndes{\end{description}}

\newtheorem{theorem}{Theorem}[section]
\newtheorem{lemma}[theorem]{Lemma}
\newtheorem{corollary}[theorem]{Corollary}

\DeclareMathSymbol{\varnothing}{\mathord}{AMSb}{"3F} 
\begin{document}

\renewcommand{\theequation}{\thesection$\cdot$\arabic{equation}}

\title{Dynamical Diophantine approximation}

\author{Ai-Hua Fan} \author{J\"{o}rg Schmeling} \author{Serge Troubetzkoy}

\address{A.\ H.\ Fan: LAMFA, UMR 4160, CNRS, University of Picardie, 33 Rue Saint
Leu, 80039 Amiens, France} \email{ai-hua.fan@u-picardie.fr}
\urladdr{http://www.mathinfo.u-picardie.fr/fan/}

\address{J.\ Schmeling: Mathematics Centre for Mathematical Sciences,
Lund Institute of Technology, Lund University Box 118 SE-221 00 Lund, Sweden}
\email{joerg@maths.lth.se}
\urladdr{http://www.maths.lth.se/matematiklth/personal/joerg/}

\address{S.\ Troubetzkoy: Centre de physique th\'eorique\\
F\'ed\'eration de re\-cher\-ches des unit\'es de math\'ematique de Marseille\\
Institut de math\'ematiques de Luminy and\\
Universit\'e de la M\'editerran\'e\\
Luminy, Case 907, F-13288 Marseille Cedex 9, France}
\email{troubetz@iml.univ-mrs.fr}
\urladdr{http://iml.univ-mrs.fr/{\lower.7ex\hbox{\~{}}}troubetz/}

\date{}
\subjclass{}
\begin{abstract}
Let $\mu$ be a Gibbs measure of the doubling map $T$ of the circle.
For a $\mu$-generic point $x$ and a given sequence $\{r_n\} \subset
\R^+$, consider  the intervals $(T^nx - r_n \pmod 1, T^nx + r_n
\pmod 1)$. In analogy to the classical Dvoretzky covering of the
circle we study the covering properties of this sequence of
intervals. This study is closely related to the local entropy
function of the Gibbs measure and to hitting times for moving
targets. A mass transference principle is obtained for Gibbs
measures which are multifractal. Such a principle was shown by
Beresnevich and Velani \cite{BV} only for monofractal measures. In
the symbolic language we completely describe the combinatorial
structure of a typical relatively short sequence, in particular we
can describe the occurrence of ''atypical'' relatively long words.
Our results have a direct and deep number-theoretical interpretation
via inhomogeneous diadic diophantine approximation by numbers
belonging to a given (diadic) diophantine class.
\end{abstract}
\maketitle

\pagestyle{myheadings}

\markboth{DYNAMICAL DIOPHANTINE APPROXIMATION}{AI-HUA FAN, J\"{O}RG
SCHMELING AND SERGE TROUBETZKOY}

\section{Introduction}\label{sec1}

Let $(X, d)$ be a complete metric space. Given a sequence
$\{x_n\}_{n\ge1}$ of points in $X$ and a sequence $\{r_n\}_{n\ge 1}$
of positive numbers we define
\begin{eqnarray*}
     I(\{x_n\},\{r_n\}) :
     &=& \limsup_{n\to \infty} B(x_n, r_n),\\
     \qquad F(\{x_n\},\{r_n\}):
     &=& X\setminus I(\{x_n\},\{r_n\})
\end{eqnarray*}
  where $B(x_n, r_n)$ denotes the ball of center $x_n$ with radius
  $r_n$. By diophantine approximation we mean the study of the sets
  $I(\{x_n\},\{r_n\})$ and $F(\{x_n\},\{r_n\})$.

Classic diophantine approximation is a special case.
Let
$X=\mathbb{S}^1=\mathbb{R}/\mathbb{Z}$ be the unit circle  equipped with the metric
$$\|x-y\| = \inf_{k\in\mathbb{Z}}|(x-y)-k|.$$ Let
$\{x_n\}=\{n\alpha \pmod 1\}$ be the orbit of the irrational
rotation determined by an irrational number $\alpha$. Then $0\in
I(\{n\alpha\},\{r_n\})$ means $\|\alpha n\|<r_n$ holds for an
infinite number of $n$'s. This is nothing but the homogeneous
diophantine approximation of $\alpha$. More generally $y\in
I(\{n\alpha\},\{r_n\})$ means $\|\alpha n -y\|<r_n$ holds for an
infinite number of $n$'s. This is what is called  inhomogeneous
diophantine approximation.
 In \cite{FS}, based on the results in \cite{ST}, both
$I(\{n \alpha\},\{r_n\})$ and $F(\{n\alpha \},\{ r_n\})$ have been
analyzed for an irrational number $\alpha$ when $r_n=n^{-\kappa}$.
The case for general sequence $\{r_n\}$ has been studied in
\cite{FW2}.

Another special case is the dynamical Borel-Cantelli lemma or
shrinking target problem. Consider a measure preserving map $T$.
A shrinking target is a sequence of balls with decreasing radius and
with centers fixed or moving (more generally, other forms than balls are
also allowed).  The question is to study the set of orbits $T^n x$
(or equivalently of the initial points) which
hit the target or equivalently which are well approximated by the
target, see for example \cite{HV} and the references therein.

There is another well studied case.  Consider  an i.i.d.\ sequence
$\{x_n\} \subset \S^1$ uniformly distributed on the unit circle
$\S^1$ with respect to Lebesgue measure, a decreasing sequence of
positive numbers $\{\ell_n\} \subset \R^+$  and the associated
random intervals $(x_n - \ell_n/2 \pmod 1, x_n + \ell_n/2 \pmod 1)$
(i.e.\ $r_n=\ell_n/2$ in the above terminology). Since $\{x_n\}$ are
independent and uniformly distributed, the Borel--Cant\-elli Lemma
assures that almost surely (a.s. for short) we have $I(\{x_n\},
\{r_n\})= \mathbb{S}^1$ except for a set of null Lebesgue measure,
i.e.\ Lebesgue a.e.\ point in $\S^1$ is covered infinitely often by
the intervals with probability one if and only if $\sum_{n=1}^\infty
\ell_n=\infty$. Moreover $\sum_{n=1}^\infty \ell_n<\infty$ implies
that Lebesgue a.e.\ point in $\S^1$ is covered finitely often with
probability one. In 1956, Dvoretzky observed  the possibility that
{\em all} points in $\S^1$ are covered infinitely often with
probability one for some slowly decreasing sequence $\{\ell_n\}$
\cite{D}. In 1972, Shepp obtained a necessary and sufficient
condition for all points in $\S^1$ to be covered infinitely often
with probability one \cite{S2}:
\[
\sum_{n=1}^\infty\frac{1}{n^2}\exp(\ell_1+\cdots+ \ell_n) = \infty.
\]
This condition is satisfied for example by $\ell_n = \frac{1}{n}$.
Important contributions were made by J.P.\ Kahane, P.\ Billard,
P.\ Erd\'{o}s, S. Orey, B.\ Mandelbrot et al. See Kahane's book
\cite{K} for a full history and a complete reference up to 1985 and
see \cite{BF,F1,F2,FK,FW1,JS} for more recent developments.


In the present work, we consider the dynamics defined by the angle
doubling map on the circle.   We shall consider a generic orbit
$\{x_n\}=\{T^n x\}$ of this map relative to a Gibbs measure. Recall
that  the doubling map $T:\S^1 \to \S^1$ is defined by $$Ts = 2s
\pmod 1.$$
We are interested in the quantity
 $$\|T^nx - y \| = \| 2^nx - y \| < r_n.$$
This is diadic diophantine approximation, homogeneous in the case
$y=0$ and inhomogeneous in the case $y \ne 0$. The sets
$I(\{x_n\},\{r_n\})$ and $F(\{x_n\},\{r_n\})$ are respectively the
sets of $y$ which are well aproximable or badly approximable with
speed $r_n$. In other words $I$ is the set of points obeying a
diophantine equation with speed $r_n$.  Our theorems are similar to
Jarnik type results in number theory. For $\kappa
> 0$ consider the special sequence $r_n = \frac{1}{n^{\kappa}}$.
Write
$$\I= (T^ns - r_n \pmod 1, T^ns + r_n \pmod 1).$$
 For $s \in \S^1$ let
 \begin{eqnarray*}
\In \hspace{-0.2cm} &:= \hspace{-0.2cm} & \bigcap_{N=1}^{\infty}
\bigcup_{n = N}^\infty \I  = \left \{t\in\S^1 \, :\,
\sum_{n=0}^\infty\1_{\I}(t)=\infty \right \},\\
\Fn \hspace{-0.2cm} &:= \hspace{-0.2cm} & \bigcup_{N=1}^{\infty}
\bigcap_{n = N}^\infty \I^c = \left\{t\in\S^1 \, :\,
\sum_{n=0}^\infty\1_{\I}(t)<\infty \right\}.
\end{eqnarray*}
 The following decomposition  is obvious:
$$\S^1 = \Fn \cup \In,\quad\Fn \cap \In =
\emptyset.
$$

It is easy to see by definition that if the orbit of $s$ is dense,
then $\In$ is a residual set, in particular, $\In\not=\emptyset$. It is
the case for a typical point $s$ relative to an ergodic measure with
full support. However, as we will see, it is possible for
$\Fn=\emptyset$ for typical points. Let $\nu_\phi,\nu_\psi$ be two
$T$-invariant probability Gibbs measures on $\S^1$ associated to
normalized H\"{o}lder potentials $\phi$ and $\psi$ (i.e.\ the
pressures of $\phi$ and $\psi$ are equal to zero). The measure
$\nu_\phi$ will be used to describe the randomness and the measure
$\nu_\psi$ to describe sizes of sets.

Let
\begin{eqnarray*}
\kappa_{\phi,\psi,\S^1} : &=& \sup \left\{ \kappa: \nu_{\psi}(\In)  = 1 \hbox{ for } \nu_\phi-a.e.\ s \right\},\\
\kappa^F_{\phi,\S^1} :&=& \sup \left\{ \kappa: \Fn = \emptyset \hbox{ for } \nu_\phi-a.e.\ s \right\}.
\end{eqnarray*}
We are interested in the following questions:  \\
\indent \ ({\bf Q1}) \ {\em How to determine the critical value
$\kappa_{\phi,\psi}$? More precisely when is $I^\kappa(s)$ of full
$\nu_\psi$-measure for $\nu_\phi$-almost every $s$?}\\
\indent \ ({\bf Q2}) \ {\em How to determine the critical value
$\kappa^F_{\phi,\S^1}$? More precisely when is $I^\kappa(s)$ equal to
$\mathbb{S}^1$ for $\nu_\phi$-almost every  $s$ ?}\\
\indent \ ({\bf Q3}) \ {\em What are the Hausdorff dimensions
$\dim_H(\Fn)$, $\dim_H(\In)$ for $\nu_\phi$-almost every $s$ ?}

Our answers to these questions are stated in the following theorems.
Let
\begin{eqnarray*}
   e^- & = &\inf_{\nu: {\rm invariant}}
     \int(-\phi)d\nu,\qquad\\
      e_{\max} &=&
     \int(-\phi)d\mbox{\rm Leb},\qquad\\
       e^+&=&\sup_{\nu: {\rm invariant}}
     \int(-\phi)d\nu
\end{eqnarray*}
where $e_-$ and $e_+$ are respectively the minimal and
maximal local entropy of $\nu_\phi$. Let $E(t)$ be the entropy spectrum of
$\nu_\phi$, which is defined by
$$
     E(t) = \dim_H \left\{y: \lim_{r\to o} \frac{\log \nu_\phi((y-r,y+r))}{\log
     r}=t\right\}.
$$
It is well known that $E(t)$ is continuous on $[e^-,e^+]$, strictly
concave and real analytic in $(e^-,e^+)$ (see \cite{Pesin}).

\begin{theorem}\label{thm1.1} The critical value  $\kappa_{\phi,\psi, \S^1}$ satisfies
$$
      \kappa_{\phi,\psi, \S^1} = \frac{1}{\int (-\phi) d\nu_\psi}.
$$
\end{theorem}

Notice that the integral $\int (-\phi) d\nu_\psi$ is nothing but the
conditional entropy of $\nu_\phi$ relative to $\nu_\psi$. The
theorem says that for $\nu_\phi$-a.e $s$ the set $I^\kappa(s)$
supports the Gibbs measure $\nu_\psi$ if $\kappa$ is small enough so
that $\int (-\phi) d\nu_\psi <\frac{1}{\kappa}$. Also notice that
for fixed $s$, the question whether $\nu_{\psi}(I^\kappa(s))=1$ is
the  shrinking target problem or dynamical Borel-Cantelli
lemma (see \cite{HV}).

\begin{theorem}\label{thm1.2}  The critical value  $ \kappa_{\phi,\mathbb{S}^1}^F$ satisfies
$$
      \kappa_{\phi,\mathbb{S}^1}^F = \frac{1}{e_+}.
$$
\end{theorem}

The theorem says that if $\kappa$ is so small that
$e^+<\frac{1}{\kappa}$, then $I^\kappa(s) = \mathbb{S}^1$ or
equivalently $F^\kappa(s)= \emptyset$ for $\nu_\phi$-a.e.\ $s$. This
is the counterpart of the Kahane-Billard-Shepp condition for the
random Dvoretzky covering.

\begin{theorem}\label{thm1.3} For $\nu_\phi$-a.e.\ $s$ we have
$$
      \dim_H F^\kappa(s)  = \begin{cases}
                           1 & \mbox{\rm if} \ \ \frac{1}{\kappa}\le e_{\max}\\
                           E(\frac{1}{\kappa}) & \mbox{\rm if} \ \ \frac{1}{\kappa}> e_{\max}
                          \end{cases}.
$$
\end{theorem}

\begin{theorem}\label{thm1.4} For $\nu_\phi$-a.e.\ $s$ we have
$$
      \dim_H I^\kappa(s)  = \begin{cases}
                           \frac{1}{\kappa} & \mbox{\rm if} \ \ \frac{1}{\kappa}\le h_{\nu_\phi}
                                \\
                           E(\frac{1}{\kappa}) & \mbox{\rm if} \ \
                           h_{\nu_\phi}<\frac{1}{\kappa}<
                           e_{\max}\\
                           1 & \mbox{\rm if} \ \ \frac{1}{\kappa} \ge  e_{\max}\\
                          \end{cases}.
$$
\end{theorem}


We will transfer the problem to a similar one in a symbolic
framework.
As we  shall see, our problem is closely related to
hitting times and the later is related to local entropy.


The structure of the article is as follows.  We start in section
\ref{back}
with background on ergodic theory, symbolic dynamics, decay of
correlations, and multi-fractal analysis.  In this section we
prove a ``multi-relation'' and a variational principal which
are essential in the proofs of the main results.
In section \ref{cove} we transfer the covering problem to the
symbolic setting and relate then covering properties to hitting
time asymptotic.  In section \ref{hitt} we prove a first simple
relation between hitting times and local entropy.  This yields
the proof of the Ornstein-Weiss return time theorem in the special
case of Gibbs measures and also allows us the determine the critical
exponent $\kappa_{\phi,\psi}$. For the other exponents more
sophisticated estimates are needed.  Sections \ref{bigh} and
\ref{smallh} contain the core estimates on the probabilities of hitting
time events.  The fundamental tools relating hitting times to the
entropy spectrum are developed.  In section \ref{typ} we study the
structure of a short typical sequence.
In particular
we make a substantial improvement in the mass transference principle
\cite{BV}
to multi-fractal Gibbs states.
Section \ref{sec8} contains the results in the symbolic framework
for the full shift while section \ref{sec9} generalizes these
results to subshifts of finite type.  Finally in section \ref{sec10}
we prove the main theorems by transferring them from the shift
space.

\section{Background}\label{back}

{\it Convention.} All logarithms and exponential functions  in this
article are taken to base 2.  With this convention the notions
of entropy and dimension coincide in our setup.
\bigskip

{\it Ergodic theory.}
We need various standard definitions from
ergodic theory: the metric entropy  of an invariant measure $\nu$
denoted by $h_{\nu}$,  the notion of the Gibbs measure $\mu_{\phi}$
with respect to a potential $\phi$ and the topological entropy for
non compact sets $E$ denoted by $h_{top}(E)$. The definitions of all
these notions can be found in \cite{Pesin}.
\bigskip

{\it Symbolic dynamics.}
 We use various standard notions from
symbolic dynamics. Let $(\Sigma_2^+,\sigma)$ denote the one sided
full shift on two symbols $0,1$. For $y = (y_i)_{i \ge 0} \in
\Sigma_2^+$ we denote a cylinder set by
$$
C_n(y) := [y_0,y_1,\cdots,y_{n-1}].$$
We will denote the length of the cylinder by $|C_n(y)| = n$.
We will denote by $$\pi(y) =
\sum_{i=0}^\infty \frac{y_i}{2^{i+1}}$$
 the natural projection from $\Sigma_2^+$ to
$\S^1$. We consider the $\frac12$-metric on $\Sigma_2^+$, i.e.\ for
$x,y \in \Sigma_2^+$ let $d(x,y) = \frac{1}{2^n}$ where $n$ is the
least integer such that $x_n \ne y_n$.  The pull back of the circle
metric  $\rho(x,y) := \sum_{i = 0}^\infty \frac{|x_i-y_i|}{2^{i+1}}$
is almost equivalent in the sense that for $x \in \Sigma_2^+$ the
ratio $\diam_\rho(C_n(x))/\diam_d(C_n(x))$ is bounded from below and
above uniformly in $n$ and $x$.  Thus Hausdorff dimensions do not
change under the projection, for details see \cite{S1}.  We denote
by $\mu_{\max}$ the measure of maximal entropy for the shift. The
projection of $\mu_{\max}$ is the Lebesgue measure on the circle.

\bigskip

\subsection{ Fast decay of correlation.} \

One of the key tools in our study is fast decay of correlations.  This
is related to Ruelle's theorem on transfer operators. Recall
that for a $\alpha$-H\"{o}lder potential $\phi: \Sigma_2^+ \to
\mathbb{R}$, i.e.
$$[\phi]_\alpha := \sup_{x, y}
|\phi(x)-\phi(y)|/d(x,y)^\alpha<\infty, $$
 the transfer operator
associated to $\phi$ is defined as follows
$$
      L_\phi f(x) = \sum_{\sigma y = x} e^{\phi(y)} f(y).
$$
This operator acts on the space of continuous functions
$C(\Sigma_2^+)$ equipped with the supremum norm $\|f\|_\infty$ and on
the space of $\alpha$-H\"{o}lder continuous functions
$H_\alpha(\Sigma_2)$ equipped with the H\"{o}lder norm
$$\|\!|f\|\!|:= \|f\|_\infty +
[f]_\alpha.
$$
 The well known Ruelle theorem asserts that \cite{Ru} \\
   \indent (i) The spectral
radius $\lambda>0$ of $L_\phi: H_\alpha \to H_\alpha$ is an
eigenvalue with an strictly positive eigenfunction $h$ and there is
a probability eigenmeasure $\nu$ for the adjoint operator
$L_\phi^*$, i.e.\ $L_\phi^* \nu = \lambda \nu$.
   \\
   \indent (ii) Choose $h$ such that $\langle h, \nu \rangle :=\int h d \nu =1$. There exist
   constants $c>0$ and $0<\beta<1$ such that for any $f \in H_\alpha$
   we have
  \begin{equation}\label{ExpDecay}
        \|\lambda^{-n} L_\phi^n f - \langle f, \nu \rangle h\|\le c
        \beta^n \|\!| f\|\!|.
   \end{equation}

Let $P(\phi) = \log \lambda$ and call it the pressure of $\phi$. The
measure $\mu:= h \nu$, denoted by $\mu_\phi$, is the so-called Gibbs
measure associated to $\phi$. Assume that $\phi$ is normalized, that
is to say $\lambda=1$. The Gibbs measure $\mu$ has the Gibbs
property: there exists a constant $\gamma>1$ such that
 \begin{equation}\label{GibbsProperty}
      \frac{1}{\gamma} e^{S_n\phi(x)} \le  \mu(C_n[x]) \le \gamma
      e^{S_n\phi(x)}
 \end{equation}
holds for all $x \in \Sigma_2$ and all $n\ge 1$ where
$$S_n f (y)  := \sum_{j=0}^{n-1} f (\sigma^j y).$$

The Gibbs property (\ref{GibbsProperty}) implies  the following
quasi-Bernoulli property of $\mu_\phi$: for any two cylinders $A$
and $B$ we have
 \begin{equation}\label{QBernoulliProperty}
    \frac{1}{\gamma^3}
     \mu_\phi(A) \mu_\phi(B) \le \mu_\phi(A \cap \sigma^{-|A|} B) \le \gamma^3
     \mu_\phi(A) \mu_\phi(B).
 \end{equation}
For the first inequality take a point $x \in A \cap
\sigma^{-|A|} B$. By using three times the Gibbs property we get
\begin{eqnarray*}
    \mu_\phi(A \cap \sigma^{-|A|} B)
     \ge   \frac{1}{\gamma} 2^{S_{|A|}\phi(x) + S_{|B|}(\sigma^{|A|}
    x)}
     \ge  \frac{1}{\gamma^3}   \mu_\phi(A) \mu_\phi(B).
\end{eqnarray*}

This quasi-Bernoulli property can be generalized in the following
way.

\begin{theorem}[Multi-relation] \label{Thm-MR} Let $\mu=\mu_\phi$ be the Gibbs measure associated
to a H\"{o}lder potential function $\phi$. Let $\omega>1$ be a
sufficiently large number. For any cylinder $D_0$ and any finite
number of cylinders $D_1,\dots,D_k$ of length $n$ we have
 \begin{equation}\label{MR}
\gamma^{-3} \left(1 -  c \beta^{n})\right)^k  \le \frac{\mu
\left(D_0 \cap \bigcap_{j=1}^k  \sigma^{-[n_0+ j(n+d)]}
D_j\right)}{\prod_{j=0}^k \mu(D_j)}
       \le \gamma^3 \left(1 +  c \beta^{n})\right)^k
 \end{equation}
where $n_0 \ge |D_0|$ and $d=d(n):\lfloor \omega n\rfloor$
\mbox{\rm(}$\lfloor a\rfloor$ denoting the integral part of a real
number $a$\mbox{\rm)}.
\end{theorem}

\begin{proof} First remark that
$$
     D_0 \cap \bigcap_{j=1}^k  \sigma^{-[n_0+ j(n+d)]}
D_j= D_0 \cap \sigma^{-|D_0|} \mathcal{B}
$$
where
$$
        \mathcal{B} = \bigcap_{j=1}^k  \sigma^{-[n_0-|D_0|+ j(n+d)]}
        D_j
$$
is a finite union of disjoint cylinders, which we denote by $B_i$'s.
Applying the quasi-Bernoulli property (\ref{QBernoulliProperty}) to
$A=D_0$ and $B=B_i$ we get
\begin{equation*}
   \frac{1}{\gamma^3}
     \mu_\phi(D_0) \mu_\phi(B_i) \le \mu_\phi(D_0 \cap \sigma^{-|D_0|} B_i) \le \gamma^3
     \mu_\phi(D_0) \mu_\phi(B_i).
\end{equation*}
Sum over all $B_i$'s and we get
\begin{equation}\label{MR*}
   \frac{1}{\gamma^3}
     \mu_\phi(D_0) \mu_\phi(\mathcal{B}) \le \mu_\phi(D_0 \cap \sigma^{-|D_0|} \mathcal{B}) \le \gamma^3
     \mu_\phi(D_0) \mu_\phi(\mathcal{B}).
\end{equation}
 Notice that the invariance of $\mu_{\phi}$ implies
 $$
             \mu_\phi(\mathcal{B}) =   \mu_\phi \left( \bigcap_{j=1}^{k}  \sigma^{-[ (j-1)(n+d)]}
D_j\right).
$$
Combining this with the equation (\ref{MR*}), it suffices to prove
\begin{equation}\label{MR**}
\left(1 -  c \beta^{n})\right)^k  \le \frac{\mu \left(
\bigcap_{j=1}^k  \sigma^{-[(j-1)(n+d)]} D_j\right)}{\prod_{j=1}^k
\mu(D_j)}
       \le  \left(1 +  c \beta^{n})\right)^k.
 \end{equation}

Actually we can prove a little more. For simplicity, we will use
$\mathbb{E} f$ to denote the integral $\int f d \mu$ and write
$\|f\|_1 = \|f\|_{L^1(\mu)}$. From the inequality
$$
    \left| \mathbb{E} ( f\circ \sigma^n \cdot g)   \right|
    = \left| \mathbb{E }(f \cdot L^n g)\right|\le  \|L^n g\|_\infty \|f\|_1
$$
(applied to $g-\mathbb{E} g$ and $f$) and Ruelle's theorem, we
deduce that for non-negative H\"{o}lder functions  $g$ and $f$ we
have
$$
   \left( 1 - c \frac{\beta^n  \|\!|g
     -\mathbb{E}g\|\!|}{\mathbb{E}g}
   \right)   \le \frac{\mathbb{E} ( f \circ \sigma^n \cdot g)}{ \mathbb{E} f \mathbb{E }g}
   \le
     \left( 1 + c \frac{\beta^n  \|\!|g
     -\mathbb{E}g\|\!|}{\mathbb{E}g}
   \right).
$$
Inductively, for a finite number of functions $g_1, \cdots, g_k \in
H_\alpha$ and for integers $0=n_1<n_2<\cdots <n_{k}$ we have
 \begin{eqnarray*}
 & &
\prod_{j=1}^{k-1}
      \left( 1  -   c \frac{\beta^{n_{j+1} -n_{j}}
     \|\!|g_{j} - \mathbb{E} g_j\|\!|}{\mathbb{E}g_j}
   \right) \\
   & &  \hspace{2cm}
    \le
   \frac{\mathbb{E} \prod_{j=1}^k g_j\circ\sigma^{n_{j}}}{\prod_{j=1}^k \mathbb{E}
   g_j}
          \le
     \prod_{j=1}^{k-1}
     \left( 1 + c \frac{\beta^{n_{j+1} -n_{j}}
     \|\!|g_{j} - \mathbb{E} g_j\|\!|}{\mathbb{E}g_j}
   \right).
\end{eqnarray*}
To get (\ref{MR**}), we  apply these inequalities to characteristic
functions of cylinders $g_j = 1_{D_j}$. In fact, since all cylinders
$D_j$ have the same length $n$, we have
 $$
 \|\!|g_j\|\!| = 1 + 2^{\alpha n}, \quad
 \frac{1}{\mathbb{E} g_j}= \frac{1}{\mu(D_j)} \le \gamma 2^{n \max_x (-\phi (x))}
 $$
 (the inequality is a consequence of the Gibbs
property).  Take $d: = \lfloor \omega n\rfloor$ with a sufficiently
large $\omega$ so that $ \beta^\omega 2^{\alpha + \max(-\phi)}<1$.
Take $n_j$ such that $n_1=0$ and $n_{j+1} -n_{j} = n+ d$ for $j\ge
2$ and the equation (\ref{MR**}) follows.
\end{proof}

We will refer to this inequality as the {\em multi-relation
property} of the Gibbs measure $\mu_\phi$.

\bigskip

\subsection{Multi-fractal analysis.} \

Furthermore  we will use various notions from multi-fractal analysis
which can also be found in the reference \cite{Pesin}. The notion of
Hausdorff dimension of a set will be denoted by $\dim_H$. For a
point $y \in \Sigma_2^+$ and an invariant measure $\nu$ we denote
the {\em lower local entropy} of $\nu$ at $y$ by
\begin{equation} \label{LowerEntropy}\underline{h}_{\nu}(y)
:= \liminf_{n \to \infty} - \frac{1}{n} \log \nu (C_n(y)).
\end{equation}
 We define the {\em local entropy} $h_{\nu}(y)$ if the limit
exists. For a function $f: \Sigma_2^+ \to \R$ we denote the ergodic
sum by
$$S_m f (y)  := \sum_{j=0}^{m-1} f (\sigma^j y).$$
We denote a Gibbs measure with respect to a H\"{o}lder potential
$\phi$ by $\mu_\phi$. Without loss of generality we may assume that
the potential is normalized so that its pressure $P(\phi) = 0$. Then
\begin{equation}\label{localent}
\underline{h}_{\mu_\phi} (y) = -\liminf_{n \to \infty} \frac1n S_n \phi (y)
\end{equation}
and $h_{\mu_\phi} (y)$ satisfies a  similar relation when the limit
exists. If $\nu$ is an ergodic invariant measure then for $\nu$
a.e.\ $y$
$$h_{\mu_\phi} (y) = - \int_{\Sigma_2^+} \phi \, d\nu.$$
Furthermore if  $\nu$ is  another Gibbs measure $\mu_\psi$ then for
$\mu_\psi$ a.e.\ $y$
\begin{equation}\label{EntropyFormulus}
h_{\mu_\phi} (y) = - P'(\psi + t \phi)|_{t=0}.
\end{equation}

Multi-fractal analysis deals with the study of the entropy
spectrum
\[
E(t) := E_{-\phi}(t) := h_{\rm top}\left\{y\, :\,
h_{\mu_\phi}(y)=t\right\}.\] The following conditional variational
is well known (\cite{BSS,FF,FFW}).

\begin{theorem}[Variational principle I]\label{varI}
Let $\phi$ be a H\"{o}lder function. For any $t\in \mathbb{R}$, we
have
\begin{equation}\label{VariationalPrinciple1}
 E(t) =
\sup_{\nu: \text{ invariant}} \left\{  h(\nu): \int (-\phi) d\nu
= t \right\}.
\end{equation} We also have
\begin{equation}\label{VariationalPrinciple2}
E(t(q))  =  P(q \phi) - q P'(q \phi)=h_{\mu_{-P(q\phi)+q\phi}}
\end{equation}
where $t(q) = - P'(q \phi)$. The range of the function $t(q)$ is an
 interval $[e^-,e^+]$, possibly degenerate to a singleton.
\end{theorem}

Let us state some more useful facts concerning the variational
principle.
 The function $t(q)$ is
invertible on the interval $[e^-,e^+]$. If $t$ is not in this
interval, then there is no point $y \in \Sigma_2^+$ with local
entropy equal to $t$.  The entropy  $E(t)$ attains its maximum at
the value
$$e_{\max} = t(0) =  \int_{\Sigma_2^+} (-\phi) d\mm.$$
 We have $t(q) \le
e_{\max}$ if and only if $q \ge 0$. Furthermore
$$e^+ =  \max_{\mu:
\text{invariant}} \int (-\phi) \, d\mu, \quad
 e^- = \min_{\mu: \text{invariant}} \int (-\phi) \, d\mu.$$
The entropy spectrum is concave and real analytic in the interval
$(e^-,e^+)$. Its graph lies below the diagonal. Moreover the
interval $[e^-,e^+]$ is degenerate if and only if $\phi$ is
cohomologous to the constant $- h_{\rm top}$, i.e.\ the measure
$\mu_\phi$ is the measure of maximal entropy. In the degenerate case
we have $e^-=e^+ = h_{\rm top}$ and $E(h_{\rm top}) = h_{\rm top}$. For typical
potentials in the sense of Baire, $E(e^-) = E(e^+) = 0$.

We will need the following variational principle.

\begin{theorem}[Variational principle II]\label{var}  Let $\phi$ be a H\"{o}lder
function. For any $t\in \mathbb{R}$, we have
$$h_{\rm top} \left \{ \underline{h}_{\mu_\phi}(y) < t \right \}
= h_{\rm top} \left \{ \overline{h}_{\mu_\phi}(y) < t \right \}= \sup_{s
< t}  E(s),$$
$$h_{\rm top} \left \{ \underline{h}_{\mu_\phi}(y) \ge t \right \}
= h_{\rm top} \left \{ \overline{h}_{\mu_\phi}(y) \ge t \right \} =
\sup_{s \ge t}  E(s).$$
\end{theorem}

\begin{proof}
Let us start with the proof of the  first fact. From the trivial
fact
$$\left \{ \underline{h}_{\mu_\phi}(y) < t \right \}
 \supset \left \{ \overline{h}_{\mu_\phi}(y) < t \right \}\supset  \bigcup_{s < t}
\{h_{\mu_\phi}(y) =s\},$$ we get immediately the following
inequalities
$$h_{\rm top} \left \{ \underline{h}_{\mu_\phi}(y) < t \right \}
\ge  h_{\rm top} \left \{ \overline{h}_{\mu_\phi}(y) < t \right \}\ge
\sup_{s < t}  E(s).$$
  Since $\sup_{t < e_{\max}} E(t) = 1$ the
converse inequalities are trivial in the case $t \ge e_{\max}$. It
remains to consider the case $t < e_{\max}$. Notice that we have $
E(t) = \sup_{s < t}  E(s)$. Also notice that there exists a
positive number $q(t) > 0$ such that
$$\min_{q \ge 0} ( P(q\phi) + qt) = P(q(t)\phi) +
q(t)t =  E(t).$$ Now let $y$ be any point such that
$\underline{h}_{\mu_\phi}(y) < t$. For $q =q(t)>0$ we can apply
Equation \eqref{localent} to yield
\begin{align*}
\underline{h}_{\mu_{- P(q\phi) + q\phi}}(y) &= \liminf_{n \to
\infty} - \frac1n S_n
\big (-P(q\phi) + q\phi \big ) (y)\\
&= P(q\phi) + q \left ( \liminf_{n \to \infty} - \frac1n S_n \phi(y) \right )\\
& \le P(q\phi) + qt =  E(t).
\end{align*}
 Thus applying the mass distribution
principle (see Theorem 7.2 of \cite{Pesin}) yields $h_{\rm top} \left \{
\underline{h}_{\mu_\phi}(y) < t \right \}\le  E(t)$, which completes
the proof of the first line.

The second fact may be similarly proved. We just  point out the
following differences that
$$\left \{ \overline{h}_{\mu_\phi}(y) \ge t \right \}
\supset \left \{ \underline{h}_{\mu_\phi}(y) \ge t \right \} \supset
\bigcup_{s \ge t} \{h_{\mu_\phi}(y) = s\},$$ and that for
$t>e_{\max}$ there exists a negative number $q(t)<0$ such that $
E(t) = P(q(t)\phi) + q(t)t$.
\end{proof}

\begin{figure}[t]
\centerline{\psfig{file=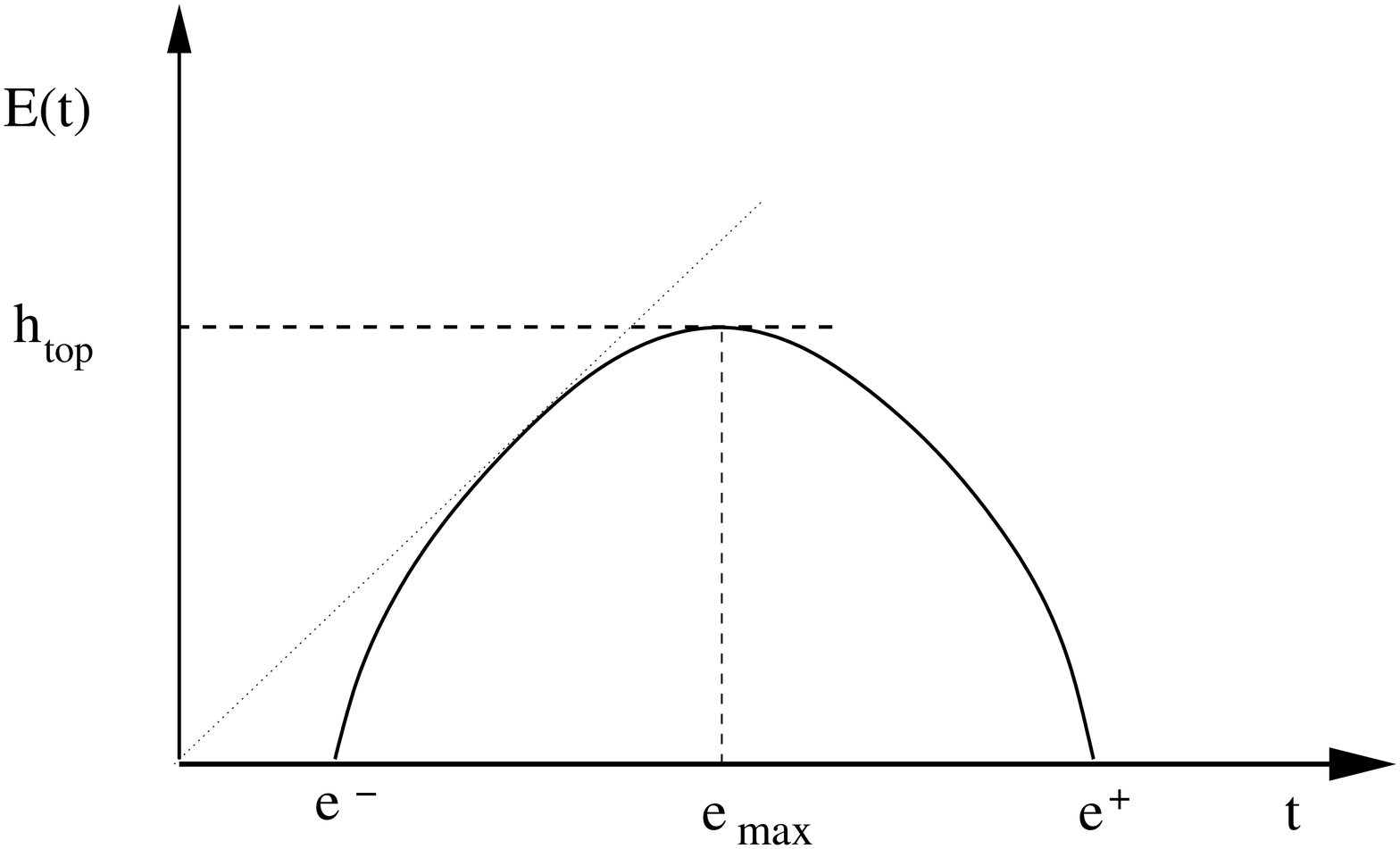,height=45mm}}
\centerline{\psfig{file=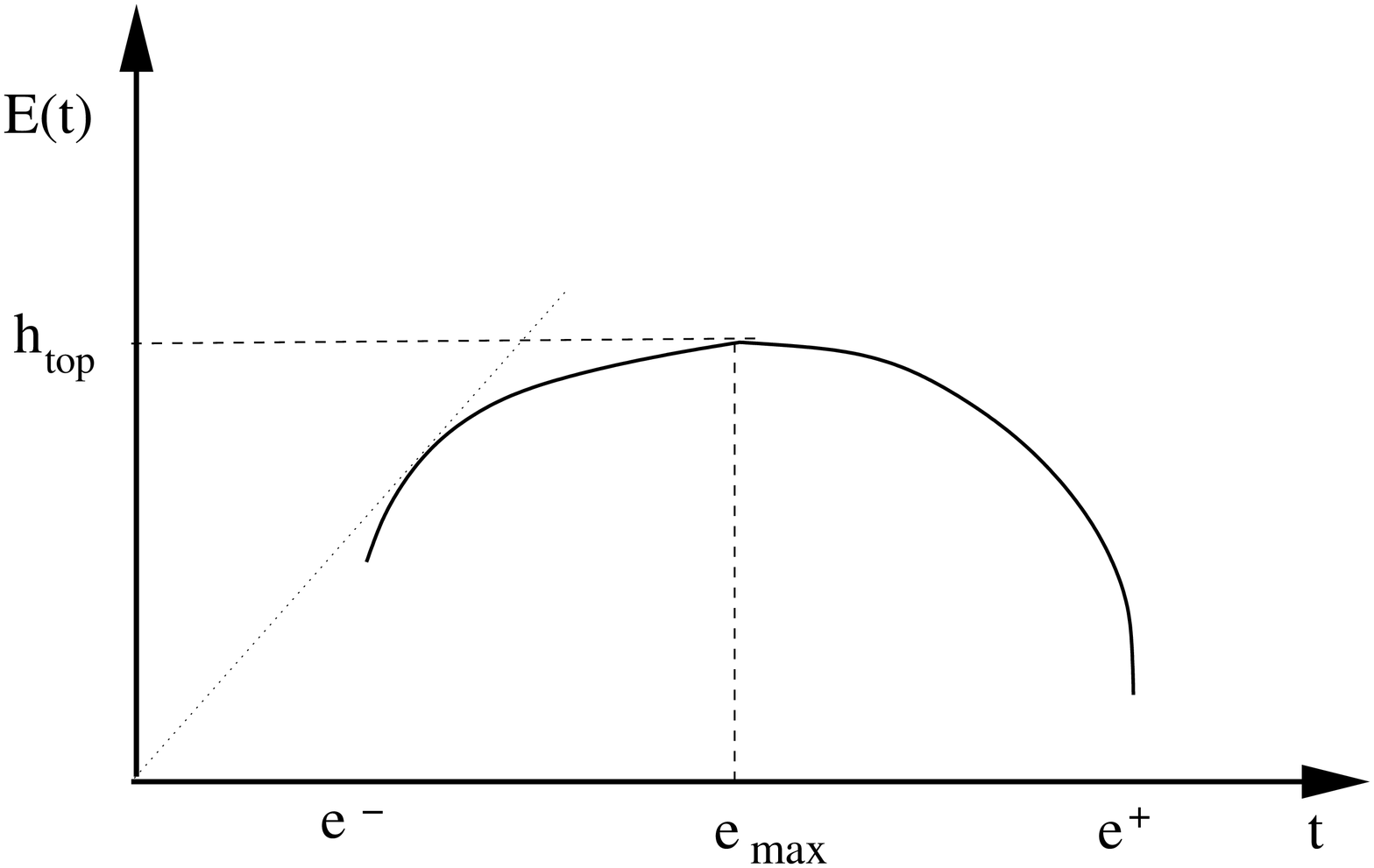,height=45mm} \quad \psfig{file=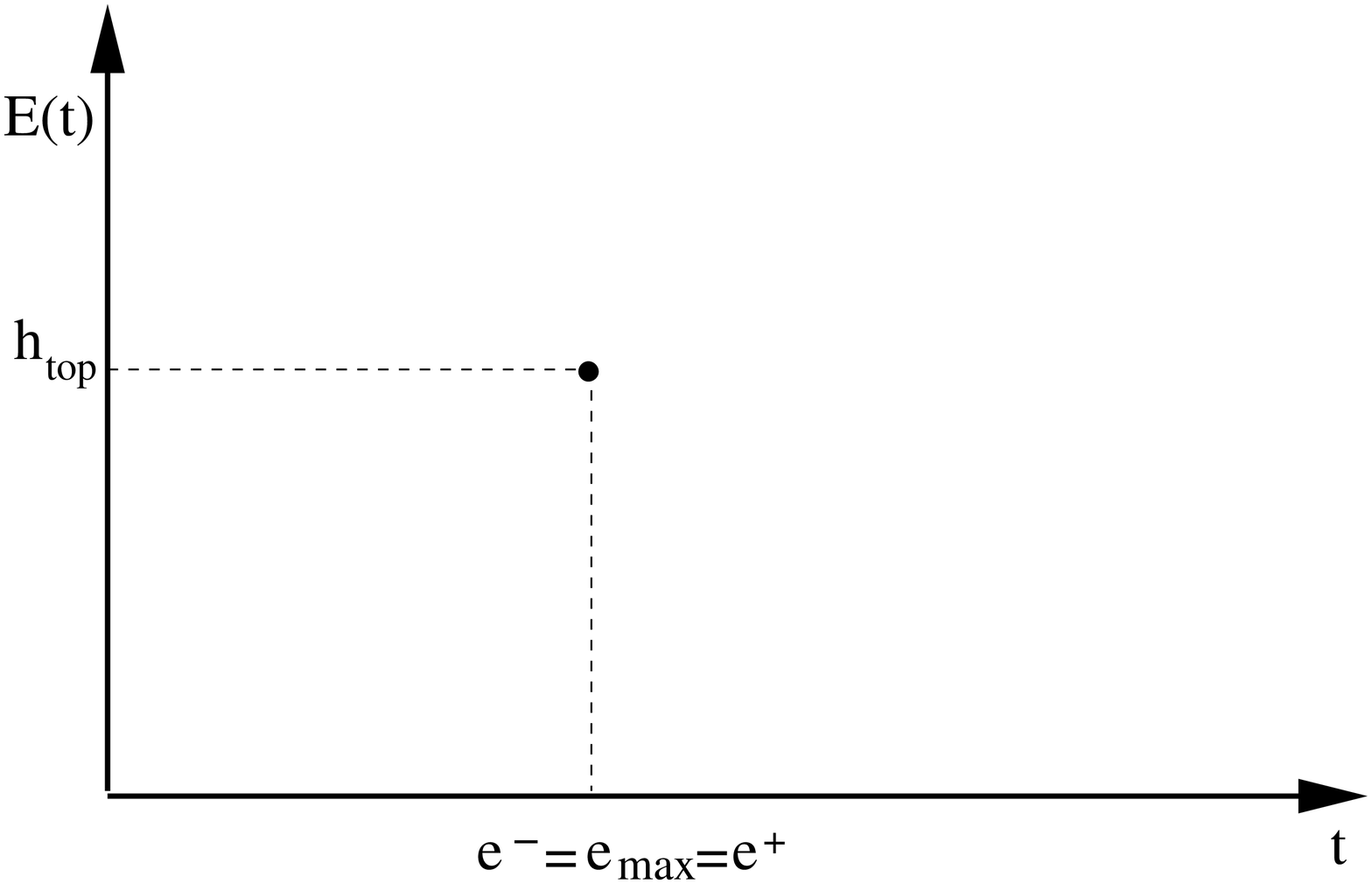,height=45mm}}
\caption{The entropy spectrum for  typical, nontypical and degenerate potentials.}\label{cov1}
\end{figure}

\setcounter{equation}{0}

\section{Covering questions are described by hitting times}\label{cove}
It is well known that the doubling map is semi-conjugate to the
shift map on $\Sigma_2^+$. As we shall see,  the initial covering
questions can be translated into  similar questions concerning the
shift map and these question are described by the hitting time that
we are going to define. We will also see that hitting times are
related to local entropy.

For $x \in \Sigma_2^+$ and $C$ a cylinder let
$$\tau(x,C) := \inf\{l \ge 1: \sigma^lx \in C\}$$
be the {\em first hitting time} of $C$ by $x$. For $x,y \in
\Sigma_2^+$ let
$$\tau_n(x,y)
:= \tau(x,C_n(y))$$
\begin{equation} \label{alpha} \alpha(x,y) := \liminf_{n \to
\infty} \frac1n \log \tau_n(x,y).
\end{equation}

Let
\begin{eqnarray*}
\Fhn &:=& \{y \in \Sigma_2^+: \ y \not \in \cap_{N=1}^{\infty} \cup_{n \ge N}  C_{\lfloor \kappa \log n\rfloor }(\sigma^n x)\},\\
\Ihn &:=& \{y \in \Sigma_2^+: \ y \in \cap_{N=1}^{\infty} \cup_{n
\ge N}  C_{\lfloor \kappa \log n\rfloor }(\sigma^n x)\}.
\end{eqnarray*}
We have the following trivial decomposition
 $$\Sigma_2^+ = \Fhn \cup
\Ihn, \qquad \Fhn \cap \Ihn = \emptyset.
$$

Suppose that $\mu_\phi,\mu_\psi$ are $\sigma$-invariant probability
Gibbs measures on $\Sigma_2^+$. Let
\begin{align*}
\kappa_{\phi,\psi,\Sigma_2^+} &:= \sup \{ \kappa:  \mu_{\psi}(\Ihn)  = 1 \hbox{ for } \mu_{\phi}-a.e.\ x \},\\
\kappa^F_{\phi,\Sigma_2^+} & := \sup \{ \kappa: \Fhn = \emptyset
\hbox{ for } \mu_\phi-a.e.\ x \}.
\end{align*}

One of our goals is to determine the values of both critical
exponents $\kappa_{\phi,\psi,\Sigma_2^+}$ and
$\kappa^F_{\phi,\Sigma_2^+}$ and the other one is to compute the
Hausdorff dimensions of $\Fhn$ and $\Ihn $. Let $$ \mathcal{O}(x)=
\{\sigma^n x: n \ge 0\}, \quad \mathcal{O}^+(x) =\mathcal{O}(x)
\setminus\{x\}.
$$

\begin{lemma}\label{lemma1.5}
There exists an integer $n_0 \ge 1$ such that $y = \sigma^{n_0} x$
(i.e.\ $y\in \mathcal{O}^+(x)$) if and only if the hitting time
sequence $\tau_k(x,y)$ is bounded.
\end{lemma}

\begin{proof}
If $y = \sigma^{n_0} x$ then it is obvious that $\tau_k(x,y) \le
n_0$ for all $k$. Conversely, suppose there is a positive constant
such that $\tau_k(x,y) \le K$. Fix an integer $1 \le t \le K$ such
that $\tau_{k_i}(x,y) = t$ holds for an infinite subsequence $k_i$.
Then $\sigma^t x \in C_{k_i}(y)$ for all $i$. Letting $i \to \infty$
we get $\sigma^t x = y$.
\end{proof}

\begin{lemma}\label{lemma2}
\begin{eqnarray*}
 \left \{y \in \Sigma_2^+: \alpha(x,y) > \frac{1}{\kappa} \right \}  \subset
\Fhn   \subset  \left \{y \in \Sigma_2^+: \alpha(x,y) \ge \frac{1}{\kappa} \right \}\cup \O^+(x),\\
\left \{y \in \Sigma_2^+   : \alpha(x,y) < \frac{1}{\kappa} \right
\} \setminus \O^+(x) \subset \Ihn \subset  \left \{y \in \Sigma_2^+:
\alpha(x,y) \le \frac{1}{\kappa} \right \}.
\end{eqnarray*}
\end{lemma}

\begin{proof}
The top left and bottom right inclusions imply one another. Let us
prove the bottom right inclusion.  Suppose $y \in \Ihn$. Then $y \in
C_{\lfloor \kappa \log n\rfloor }(\sigma^n x)$ or equivalently
$\sigma^n x \in C_{\lfloor \kappa \log n\rfloor }(y)$ for infinitely
 many $n$.
 Thus $\tau_{\lfloor \kappa
\log n\rfloor}(x,y)\le n$ for infinitely many $n$, which implies
$\alpha(x,y) \le \kappa^{-1}$.


The top right and bottom left inclusions imply one another. So, it
remains to prove the bottom left inclusion. Suppose
$\alpha:=\alpha(x,y) < \kappa^{-1}$ and $y \not \in \O^+(x)$. Take
$\e>0$ such that $\kappa<\frac{1}{\alpha+\e}. $ By the definition of
$\alpha := \alpha(x,y)$,  there is a subsequence $k_i$ such that
$\log \tau_{k_i}(x,y) \le (\alpha +\e) k_i$, i.e.\ $k_i \ge
\frac{\log \tau_{k_i}(x,y)}{\alpha + \e}$. The definition of
$\tau_{k_i}(x,y)$ implies that
$$
\sigma^{\tau_{k_i}}x \in C_{k_i}(y)\subset C_{\left \lfloor
\frac{\log \tau_{k_i}}{\alpha + \e} \right \rfloor}(y) \subset
C_{\lfloor \kappa \log \tau_{k_i}\rfloor}(y).
$$ Since $y \not \in \O^+(x)$ the previous lemma yields that $\tau_{k_i}$
is not bounded.  Thus
 $\sigma^nx \in C_{\lfloor \kappa
\log n \rfloor}(y)$ or equivalently $y \in C_{\lfloor \kappa \log n
\rfloor}(\sigma^n x)$ for   infinitely many $n = \tau_{k_i}$.

\end{proof}

We should point out that points $y$ on the orbit $\O^+(x)$
have the property that $\alpha(x, y)=0<1/\kappa$, but they are not
necessarily contained in $\Ihn$. For example, if $x$ is an
eventually periodic point but not periodic and if $y$ is on the
orbit $\O^+(x)$ but not in the cycle of $x$, then $y\not\in \Ihn$.
However, for $\mu_\phi$-almost all $x$, we have the following
situation.

\begin{lemma}\label{OW}
For $\mu_{\phi}$ a.e.\ $x$, we have    $\O(x) \subset \Ihn$ if
$\frac{1}{\kappa}  > h_{\mu_{\phi}}$ and  $\O(x) \subset \Fhn$ if
$\frac{1}{\kappa}  < h_{\mu_{\phi}}$.
\end{lemma}

\begin{proof}
Let $y \in \O(x)$ where $x$ is not eventually periodic.   Then there
exists a unique integer $n_0\ge 0$ such that $y = \sigma^{n_0}x$.
Define the hitting time after $n_0$
by
$$ \tau^{(n_0)}_n (x, y) := \inf \{k > n_0: \sigma^kx \in C_n(y) \} = \tau_n(\sigma^{n_0} x, y)+n_0.$$
Since $y \not\in \mathcal{O}^+(\sigma^{n_0}x)$) Lemma \ref{lemma1.5}
implies that $\tau^{(n_0)}_n (x,y) \to \infty$ as $n \to \infty$.
Let
\begin{equation}\label{AlphaN0} \alpha^{(n_0)}(x,y) = \liminf_{n \to
\infty} \frac1n \log \tau_n^{(n_0)}(x,y).
\end{equation}
Hence $$y \in \Ihn\ \  \mbox{\rm if} \ \  \alpha^{(n_0)}(x,y) <
\frac{1}{\kappa},\quad \mbox{\rm and} \quad y \in \Fhn  \ \
\mbox{\rm if}\  \ \alpha^{(n_0)}(x,y) > \frac{1}{\kappa}.$$ Now
$$\alpha^{(n_0)}(x,y) = \alpha(y,y) =
\alpha(\sigma^{n_0}x,\sigma^{n_0}x).$$ Thus applying the
Ornstein-Weiss return time theorem \cite{OW} yields that
$\alpha(x,x) = h_{\mu_{\phi}}$ for $\mu_{\phi}$-a.e.\ $x$. Finally
the invariance of $\mu$ implies that
$\alpha(\sigma^{n}x,\sigma^{n}x) = h_{\mu_{\phi}}$ for $\mu_{\phi}$
a.e.\ $x$ and for  all $n$.
\end{proof}

\setcounter{equation}{0}
\section{Hitting time and local entropy: basic relation}\label{hitt}

As Lemmas~\ref{lemma2} and~\ref{OW}  show, we have to study the
hitting time $\alpha(x,y)$ of the Gibbs measure $\mu_\phi$.  We will
show   that the hitting time is related to the local entropy. Local entropy
have been well studied in the literature.

In this section, we start with a basic relation between hitting
times and local entropy. This allows us to
compute the critical value $\kappa_{\phi,\psi,\Sigma_2^+}$.

Let us first introduce a generalized notion of local entropy.
Let $(C_{n})$ be a sequence of (arbitrary) cylinders with length
$|C_n| = n$. We
define the {\em lower local entropy of the sequence $(C_{n})$} by
\begin{equation}\label{LocalEntropySequence}
\underline{h}_{\mu_\phi}({\{C_{n}\}}) :=
\liminf_{n\to\infty}-\frac{\log \mu_\phi(C_{n})}{n}. \end{equation}

\subsection{Basic relation}\
We have the following basic relation between local entropy and
the hitting times.

\begin{theorem}\label{Cha} Suppose that $\mu_{\phi}$
is a Gibbs measure associated to a H\"{o}lder potential $\phi$ and
that $(C_{n})$ is a sequence of (arbitrary) cylinders of length
$n$. Then for $\mu_\phi$ a.e.\ $x$ we have
\begin{equation}\label{cha2}
\liminf_{n\to\infty}\frac{\log\tau(x,
C_n)}{n}=\underline{h}_{\mu_\phi}({\{C_{n}\}})
\end{equation}
\end{theorem}

\begin{proof}
A special case of this theorem was proven by Chazottes \cite{Ch}.
The proof follows the idea of Chazottes closely. We include it for
completeness.

Let $\tau_n(x) :=  \tau(x,C_{n})$.
Note that the Gibbs property implies $\mu_\phi(C_{n}) \to 0$.  Fix
$\e
> 0$ and let
\begin{align*}
A_n & := \big \{x \in \Sigma_2^+: \ \tau_n(x) \mu_\phi(C_{n})< 2^{-\e n} \big  \},\\
B_n & := \big \{x \in \Sigma_2^+: \ \tau_n(x) \mu_\phi(C_{n})> 2^{\e
n} \big \}.
\end{align*}
We will prove that  
$$ \sum \mu_\phi(A_n \cup B_n) \le
\sum \mu_\phi(A_n) + \sum \mu_\phi(B_n)< \infty.$$
  Once we have shown this we apply the first part
of the Borel-Cantelli lemma to conclude the proof.

First consider the series $\sum \mu_\phi(A_n)$, which is simpler to
handle. We have
  $$A_n
\subset A_n^0 \cup \cdots \cup A_n^m$$ where $$A_n^i := \{ x \in
\Sigma_2^+: \ \sigma^ix \in C_{n}\}, \quad m = \lfloor 2^{-\e
n}/\mu_\phi(C_{n}) \rfloor.$$
 Since $\mu_\phi(A_n^i) = \mu_\phi(A_n^j)
= \mu_\phi(C_{n})$, this yields $$\mu(A_n) \le \left( \frac{2^{-\e
n}}{\mu_\phi(C_{n})} + 2 \right) \mu_\phi(C_{n}) \le 2^{-\epsilon n}
+ 2 \mu_\phi(C_{n}).$$ Now we distinguish two cases:
$\underline{h}_{\mu_\phi}({\{C_{n}\}})>0$ and
$\underline{h}_{\mu_\phi}({\{C_{n}\}})=0$. In the first case,
$\mu_\phi(C_{n})$ decays exponentially fast, so that $\sum
\mu_\phi(C_{n})<\infty$, then $\sum \mu_\phi(A_n)<\infty$. In the
second case, since $\mu_\phi(C_{n}) \to 0$, we can find some
subsequence $n_k$ such that $\sum_k \mu_\phi(C_{n_k})<\infty$ so
that $\sum_k \mu_\phi(A_{n_k})<\infty$. So
$$
   \liminf_{n\to\infty}\frac{\log\tau(x,
C_n)}{n}\le \liminf_{k\to\infty}\frac{\log\tau(x, C_{n_k})}{n_k}=0.
$$

Now we turn to the analysis of the series $\sum \mu_\phi(B_n)$.
Choose a big $\omega > 0$
and $d := d(n):= \lfloor \omega n \rfloor$. Let
 $$B_n^i := \{x:  \sigma^{i(n+d)}x
\not\in C_n \}, \quad m := \lfloor 2^{\e n}/\mu_\phi(C_n)(n + d)
\rfloor - 1.$$ Thus
$$B_n \subset B_n^0 \cap \cdots \cap B_n^m = \bigcup_{D_0,\dots,D_m} D_0 \cap \sigma^{-(n+d)}D_1 \cap
\cdots \cap \sigma^{-m(n+d)}D_m$$ where the $D_i$ are cylinders (not
necessarily distinct) of length $n$ disjoint from $C_n$. Thus, by
the multi-relation property, we get
\begin{align*}
\mu_\phi(B_n) & \le \sum_{D_0,\dots,D_m} \mu_\phi( D_0 \cap
\sigma^{-(n+d)}D_1 \cap
\cdots \cap \sigma^{-m(n+d)}D_m)\\
& \le ( 1 + c\beta^d)^m \sum_{D_0,\dots,D_m}   \prod_{i=0}^m \mu_\phi(\sigma^{-i(n+d)}D_i)\\
&\le [(1 + c\beta^d)(1 - \mu_\phi(C_n))]^{m+1}
\\
& \le \left ( 1 - \frac{\mu_\phi(C_n)}2  \right )^{m+1}  \\
& \le e^{-(m+1) \mu_\phi(C_n)/2}\\
& \le e^{-2^{\epsilon n -1}/(n+d)}.
\end{align*}
\end{proof}

\begin{corollary}
For any $y \in \Sigma_2^+$ and  for $\mu_\phi$ a.e.\ $x$

$$\alpha(x,y) = \hb_{\mu_\phi}(y).$$
\end{corollary}
An application of Fubini's Theorem yields
\begin{corollary}\label{fub}
Let $\nu$ be a probability measure on $\Sigma_2^+$. Then for
$\mu_\phi\times\nu$ a.e.\ $(x,y)$ we have
\[
\alpha(x,y)=\hb_{\mu_\phi}(y).
\]
\end{corollary}

The hitting time $\alpha(x,x)$ is what we called the return time.
The following result due to Ornstein and Weiss \cite{OW} concerning the return
time is well known and holds for all ergodic measures. For Gibbs
measures, it can be similarly proved as the above theorem.

\begin{corollary}\label{return}
For $\mu_\phi$ a.e.\ $x$ we have
\[
\alpha(x,x)=\alpha(\sigma^k x,\sigma^k
x)=\hb_{\mu_\phi}(x)=h_{\mu_\phi}\qquad (\forall k\ge 1).
\]
\end{corollary}

\subsection{Determination of $\kappa_{\phi,\psi,\Sigma_2^+}$}\

Recall that $-\int
\phi d\mu_{\psi}$ is nothing but the conditional entropy of
$\mu_{\phi}$ relative to $\mu_{\psi}$.
As a direct consequence of Lemma \ref{lemma2} and Chazottes' theorem, we
get immediately the following critical value.

\begin{theorem}\label{nu} Let $\phi$ and $\psi$ be H\"{o}lder
functions on $\Sigma^+_2$. We have
 $$\kappa_{\phi,\psi} = \frac{1}{-
\int_{\Sigma_2} \phi \, d\mu_{\psi}} = - \frac{1}{\frac{d}{dt}P(\psi
+ t \phi)|_{t=0}}.$$
\end{theorem}

\begin{proof}
Suppose that $\mu_\phi$ and $\mu_\psi$ are ergodic Gibbs measures
with $P(\phi) = P(\psi) = 0$. Corollary \ref{fub} implies that
for $\mu_{\phi} \times \mu_{\psi}$ a.e.\ $(x,y)$
$$\alpha(x,y) = h_{\mu_{\phi}}(y)  = - \int_{\Sigma_2} \phi \, d\mu_\psi
= - \frac{d}{dt}P(\psi + t \phi)|_{t=0}.$$
Thus applying Lemma \ref{lemma2} yields the assertion of
the theorem.
\end{proof}

\setcounter{equation}{0}
\section{ Big hitting probability and Study of $\Fhn$}\label{bigh}

We will give answers to question ({\bf Q2}) and to the part of question
({\bf Q3}) concerning $\Fhn$.

\subsection{Big hitting probability}\

Heuristically points of small local entropy (i.e.\ large ``local
measure'') are hit with big probability.
More precisely we have

\begin{lemma}[Big hitting probability]\label{lemma2'}
Let $K := 2^{hn}$. Fix $L$ cylinders  $C_1, \cdots C_L$ of length
$n$ satisfying $\mu_\phi(C_i) \ge 2^{-(h - \gamma)n}$. Then
\begin{align*}
\mu_{\phi}\{x: \, \exists C \in \{C_i\} \text{ such that }
\tau_n(x,C) > K\} \le  2^{-\lambda n}
\end{align*}
for any positive $\lambda$ for sufficiently large $n$.
\end{lemma}

\begin{proof}
We have $L$ possibilities for the cylinder $C$. Let $m := \lfloor
K/(1+ \omega)n \rfloor - 1$. Fix a choice $C$ from these $L$
cylinders and let $D_0,\dots,D_m$ denote any cylinders of length $n$
(possibly with repetition), which are disjoint from $C$. Choose $\omega
> 0$ so that $\beta^\omega < 2^{e^+}$. Let $d := d(n):= \lfloor
\omega n \rfloor$.

For a fixed $C$, let $G_C$ be the set of points in $\Sigma_2^+$ in
which the chosen cylinder $C$, considered as a word, does not appear
up to time $K$. In particular, it does not appear at times $n+d,
\cdots, m(n+d)$. Thus
\begin{align*}
\mu_\phi(G_C) & \le \sum_{D_0,\dots,D_m} \mu_\phi( D_0 \cap
\sigma^{-n+d}D_1 \cap \cdots \cap \sigma^{-m(n+d)}D_m).
\end{align*}
By the multi-relation property, we get
\begin{align*}
\mu_\phi(G_C)
& \le ( 1 + c\beta^d)^{m+1} \sum_{D_0,\dots,D_m}   \prod_{i=0}^m \mu_\phi(\sigma^{-i(n+d)}D_i)\\
&= \left[(1 + c\beta^d)(1 - \min_{C_i} \mu_\phi(C_i))\right]^m\\
& \le \left ( 1 - \frac{1}2 {\min_{C_i} \mu_\phi(C_i)}  \right )^m .
\end{align*}
Summing over all the $L (\le 2^n)$ possible cylinders $C$ yields
\begin{align*}
\mu_{\phi}\{x: \, & \exists C \in \{C_i\} \text{ such that } \tau_n(x,C) > K\}\\
& \le \sum_{C} \mu_\phi(G_C)\\
& \le L \left( 1 - \frac{1}{2} \min_{C_i} \mu_\phi(C_i))\right)^m\\
& \le L  \left( 1 - \frac{1}2 \min_{C_i} \mu_\phi(C_i))\right)^{2^{\gamma n}/(\min_{C_i}\mu_\phi(C_i)(1 + \omega)n)}\\
& \le const \cdot 2^{n} \cdot  \big( e^{-1/2}\big )^{2^{\gamma n}/(1+\omega)n}\\
& \le 2^{- \lambda n}
\end{align*}
for any positive $\lambda$ and sufficiently large $n$.
\end{proof}

\subsection{The set of late  hits. }\

Let us
recall that $\{y \in \Sigma_2^+: \, \alpha(x,y) \ge t \}$ is random
but $\{ y \in \Sigma_2^+: \, \hb_{\mu_\phi}(y)\ge t\}$ is
deterministic (i.e.\ independent of $x$). The following theorem is
deduced from Lemma~\ref{lemma2'} (big hitting probability) and
Corollary~\ref{fub} (Ornstein-Weiss type theorem on return times).

\begin{theorem}\label{thm5.2}
For any $t\ge 0$ and for $\mu_\phi$ a.e.\ $x$ we have
\begin{equation}\label{EarlyHitSet1}
\{y \in \Sigma_2^+: \, \alpha(x,y) \ge t \} \subset \{ y \in
\Sigma_2^+: \, \hb_{\mu_\phi}(y)\ge t\}. \end{equation}
 Moreover if
$\nu$ is any probability measure on $\Sigma_2$, then for $\mu_\phi$
a.e.\ $x$ we have
$$\{y \in \Sigma_2^+: \, \alpha(x,y) \ge t \} \stackrel{\nu}{ =} \{ y \in \Sigma_2^+: \,
\hb_{\mu_\phi}(y)\ge t\}.$$
\end{theorem}

\begin{proof} The case $t=0$ is trivial. Assume $t>0$.  Let
$$
     H_{\ge t}(x) = \{y \in \Sigma_2^+: \, \alpha(x,y) \ge t
     \},\quad
     E_{\ge t} = \{ y \in \Sigma_2^+: \,
\hb_{\mu_\phi}(y)\ge t\}.
$$
By definition, we have $$
  H_{\ge t}(x) = \bigcap_{\e>0} \liminf_{n \to \infty} H_{n, \e}(x)
$$
with $H_{n, \e}(x)= \{y: \tau_n(x, y) \ge 2^{(t-\e)n}\}$, and
$$
  E_{\ge t} = \bigcap_{\e>0} \liminf_{n \to \infty} E_{n, \e}
$$
with $E_{n, \e}(x)= \{y: \mu_\phi(C_n(y)) \le 2^{-(t-2\e)n}\}$. Thus
it remains to prove that for $\mu_\phi$-a.e.\ $x$ there exists $n(x) >
0$ such that
$$
        H_{n,\e}(x) \subset E_{n, \e}  \quad \forall n \ge n(x).
$$
Equivalently
$$
        E_{n,\e}^c \subset H_{n, \e}^c(x)  \quad \forall n \ge
        n(x).
$$
Notice that $E_{n,\e}^c$ is the union of all $n$-cylinders $C$ such
that $\mu_\phi(C) >2^{-(t-2\e)n}$. Let $\mathcal{C}_{n, \e}$ be the
set of all these cylinders. Applying Lemma \ref{lemma2'} to $\{C_1,
\cdots, C_L\}:=\mathcal{C}_{n, \e}$ leads to
$$
   \sum_n \mu_\phi\{ x\in \Sigma_2: \exists C \in \mathcal{C}_{n, \e} \ \mbox{\rm s.t.}\
     \tau_n(x, C) \ge 2^{(t-\e)n}\}<\infty.
$$
So, by the Borel-Cantelli lemma, for $\mu_\phi$-a.e.\ $x$, for large
$n$ and for all $C \in \mathcal{C}_{n, \e}$ we have $\tau_n(x, C)<
2^{(t-\e)n}$, i.e.\ $C \subset H_{n,\e}^c(x)$. This proves the first
assertion.

To prove the second assertion, it suffices to show that for
$\mu_\phi$-a.e.\ $x$ we have
$$
    \nu\{y\in \Sigma_2:    \underline{h}_{\mu_\phi}(y)\ge t, \alpha(x, y)<t\}=0.
$$
Let
$$
    E=\{(x,y): \alpha(x, y) = \underline{h}_{\mu_\phi}(y)\},
    \quad E_x=\{y: \alpha(x, y) = \underline{h}_{\mu_\phi}(y)\}.
$$
By Corollary~\ref{fub}, we have $\mu_\phi \times \nu(E)=1$. Then
Fubini's theorem asserts that for $\mu_\phi$-a.e.\ $x$ we have
$\nu(E_x)=1$, i.e.
$$
     \nu(E_x^c) = \nu\{y: \alpha(x, y) \not= \underline{h}_{\mu_\phi}(y) \}=0.
$$
We conclude by noticing
$$
    \{y: \underline{h}_{\mu_\phi}(y) \ge t,  \alpha(x, y)<t\}
    \subset E_x^c.
$$
\end{proof}

We should point out that (\ref{EarlyHitSet1}) is equivalent to
\begin{equation}\label{EarlyHitSet2}
\{ y \in \Sigma_2^+: \, \hb_{\mu_\phi}(y)< t\} \subset \{y \in
\Sigma_2^+: \, \alpha(x,y) < t \}.
\end{equation}
This justifies our heuristics that points of small local entropy
are hit early.
We point out that the inverse inclusion of
(\ref{EarlyHitSet2}) does not hold.
 Actually for $t<e^-$,
the deterministic set $\{ y \in \Sigma_2^+: \, \hb_{\mu_\phi}(y)<
t\}$ is empty, but if $1/\kappa< t$, the random set $\{y \in
\Sigma_2^+: \, \alpha(x,y) < t \}$ contains $I^\kappa(x)$ which is a
residual set.

\subsection{\bf Computation of $\dim_H \{y: \alpha(x, y)\ge t\}$ and $\dim_H
\Fhn$}\

\begin{theorem}\label{5.3} For $\mu_\phi$-a.e.\ $x$, we have
\begin{align*}
\dim_H \big \{ y: \ \a(x,y) \ge t \big \} &= \dim_H \big \{ y: \
\hb_{\mu_\phi} \ge t \big \}. \ \text{}
\end{align*}

\end{theorem}

\begin{proof}
By the second variational principle  (Theorem~\ref{var}), there exists
an $s \ge t$ such that
\begin{equation}\label{DimA-F1}
\dim_H \{ y: \ \hb_{\mu_\phi} \ge t \big \} = \dim_H
\mu_{-P(q(s)\phi) + q(s) \phi}.
\end{equation}
Applying Corollary~\ref{fub} (with $\nu = \mu_{-P(q(s)\phi) + q(s)
\phi}$) implies that
$$
\mu_{-P(q(s)\phi) + q(s) \phi}(\{y: \ \hb_{\mu_\phi}(y) = \a(x,y) =
s\}) = 1 \ \ \mbox{\rm for}\ \mu_\phi-\mbox{\rm  a.e.}\ x.
$$
It follows that for $\mu_\phi$-a.e.\ $x$ we have
\begin{align*}
\dim_H \{ y: \ \a(x,y) \ge t \} & \ge \dim_H \{ y: \
\hb_{\mu_\phi}(y) = \a(x,y) = s \}   \\
 & \ge  \dim \mu_{-P(q(s)\phi)
+ q(s) \phi}.
\end{align*}
This, together with (\ref{DimA-F1}), implies
$$
   \dim_H \big \{ y: \ \a(x,y) \ge t \big \} \ge  \dim_H \big \{ y: \
\hb_{\mu_\phi} \ge t \big \} \quad \text{}  \mu_\phi\mbox{\rm -a.e.}
$$

  Now we turn to the
reverse inequality. Observe the following decomposition
$$\{y: \a(x,y) \ge t \} = \{\a(x,y) \ge t,  \hb_{\mu_\phi}(y) < t
\} \cup \{  \a(x,y) \ge t, \hb_{\mu_\phi}(y) \ge t \}.$$ Since
$$\dim_H \{ \hb_{\mu_\phi}(y) \ge t, \a(x,y) \ge t \} \le  \dim_H \{ \hb_{\mu_\phi}(y) \ge t \},$$
it suffices to remark that $\{y: \ \hb_{\mu_\phi}(y) < t, \a(x,y)
\ge t \} = \emptyset$  for $\mu_\phi$ a.e.\ $x$.
\end{proof}

By this theorem, Lemmas~\ref{lemma2} and \ref{OW}, and the
second variational principle (Theorem~\ref{var}) we get
\begin{theorem} For $\mu_\phi$-a.e.\ $x$ we have
\begin{eqnarray*}
h_{\rm top}(\Fhn) & = &1                     \qquad \quad  \ \mbox{\rm for} \ \frac{1}{\kappa} \le e_{\max},\\
h_{\rm top}(\Fhn) & = &h_{\mu_{q(\kappa)\phi}} \ \ \ \mbox{\rm for}\ \ \
\ e_{\max}
                                                    \le \frac{1}{\kappa} < \ e_+
\end{eqnarray*}
where $q(\kappa)$ is chosen
such that $h_{\mu_\phi}(y) = \frac1{\kappa}$ for
$\mu_{q(\kappa)\phi}$ a.e.\ $y$. We also have
\begin{eqnarray*} \Fhn         & = &
\emptyset \ (\mbox{\rm or equivalently}  \  \Ihn=\S^1 )  \mbox{ \rm
if} \ \  \frac{1}{\kappa} >
                                                 e_+,\\
\Fhn         & \not= & \emptyset \ (\mbox{\rm or equivalently}  \
\Ihn\not=\S^1 )  \mbox{ \rm if} \ \  \frac{1}{\kappa} <
                                                 e_+.
\end{eqnarray*}
\end{theorem}

Remark that the case $\frac{1}{\kappa}=e^+$ is not covered by the
theorem because $E(t)$ is not continuous at $t=e^+$. We have the
upper bound $\dim_H \mathcal{F}^{1/e_+} \le E(e^+)$. A result due
to Kahane for the random covering shows that a strict inequality may occur (\cite{K}, p.160).


\setcounter{equation}{0}

\section{Small hitting probability and upper bound of $\dim_H \{y: \alpha(x, y)\le s\}$}\label{smallh}
\subsection{Small hitting probability}\

\begin{lemma}[Small hitting probability]\label{lemma1add}
Let $K := 2^{an}, L := 2^{bn}, N := 2^{cn}$ with $a>0, b>0, c>0$.
Fix $L$ different cylinders  $C_1, \cdots C_L$ of length $n$
satisfying
$$ \mu_\phi(C_i) \le 2^{-(a + \gamma)n}.$$
 Then if
$\gamma
> \max (b-c, 0)$, for any positive $\lambda$ and sufficiently large $n$ we have
\begin{align*}
\mu_{\phi}\{x: \, \tau_n(x,C_i) \le  K \text{ for } N &\text{
different cylinders among the}\ C_i\} \le 2^{-\lambda n}.
\end{align*}

\end{lemma}

\begin{proof}  Let $S$ be  the set in question.
That $x\in S$ means there exist times $\ell_1<\ell_2<\cdots
<\ell_N<K$ and different cylinders $C_{i_1}, C_{i_2}, \cdots,
C_{i_N}$ such that
$$
    \sigma^{\ell_1}x \in C_{i_1}, \ \
    \sigma^{\ell_2}x \in C_{i_2},\ \ \cdots, \ \  \sigma^{\ell_N}x \in
    C_{i_N}.
$$
In this sequence $(\ell_k)$ of length $N$ there is a subsequence of
$N/(3n+d)$ terms, denoted $(\tau_j)$ such that $\tau_j
-\tau_{j-1}\ge 3n +d$. For example, we may take $\tau_j = \ell_{(3n
+ d)j}$. Thus $x\in S$ implies
$$
    \sigma^{\tau_1}x \in C_{j_1}, \ \
    \sigma^{\tau_2}x \in C_{j_2},\ \ \cdots, \ \  \sigma^{\tau_{N'}}x \in
    C_{j_{N'}}
$$
for $N':=N/(3n +d)$ different cylinders taken from the list $C_1,
C_2, \cdots, C_L$.
Thus to each $x \in S$ we can associate the sequences $(\tau_j)$ and $(C_{j_k})$.
Thus
$$
x \in C(x) := \bigcap\sigma^{-\tau_i}(C_{j_i})
$$
 and $S$ is covered
by the union of $C(x)$.
The multi-relation property implies that
the measure of $C(x)$ is bounded by $$ \max_{1\le
i\le L} \mu_\phi(C_i)^{N'} (1+c\beta^d)^{N'}.$$ Now, we have to
estimate the number of different (disjoint) sets $C(x)$.
First we have ${L \choose
N'}$
choices for the $N'$ different
cylinders from the list of $L$ words. Then we can choose ${K \choose
N'}$ places (i.e.\ we fix the sequence $\tau_j$) to put the chosen
words in order to
determine $C(x)$ . Finally we have $N'!$  ways to
arrange words into these $N'$ (now fixed) places.

Thus the measure of the
set in question can be majorized by
$$
{L\choose N'} {K \choose N'} \cdot N'!
\cdot\max_{C_i}\mu_\phi(C_i)^{N'}\cdot (1+c\beta^d)^{N'}. $$
This is equal to
$$ \frac{L!}{(L-N')!} \cdot\frac{K!}{(K-N')!N'!}\cdot
(\max_{C_i} \mu_\phi(C_i))^{N'}\cdot (1+c\beta^d)^{N'}.
$$
Next using the estimates
$$\frac{L!}{(L-N')!} \le L^{N'},
\quad \frac{K!}{(K-N')!N'!} \le const\cdot K^{N'}\cdot
\frac{e^{N'}}{{N'}^{N'}}$$ (the second one is implied by  Stirling's
formula), we  conclude that the measure is majorized by
\begin{align*}
& const\cdot L^{N'}\cdot K^{N'}\cdot e^{N'}\cdot N'^{-N'}\cdot
\left(2^{-(a+\gamma)n}\right)^{N'}\cdot (1+c\beta)^{N'}\\& =
const\cdot\left(2^{bn}\cdot 2^{an}\cdot e\cdot 2^{-cn}\cdot
  2^{-(a+\gamma)n}\cdot (1+c\beta^d)\right)^{N'}\\&\le
const \left ( {e\cdot (1+c\beta^d) \cdot 2^{(b-c - \gamma) n}} \right
)^{N'}.
\end{align*}
Provided $\gamma>b-c$, this is less that $2^{\lambda n}$
for any positive  $\lambda$ and sufficiently large $n$.
\end{proof}

\subsection{Upper bound of $\dim_H \{y: \alpha(x, y)\le s\}$}\

\begin{theorem}\label{cor1add}
If $h_{\mu_\phi}<s<e_{\max}$ then for $\mu_\phi$-a.e.\ $x$ we have
\begin{equation}
\label{UpperBoundIhn1}
h_{\rm top} \left\{ y\,:\,  \a(x,y)
  \le  s\right\} \le E(s).
  \end{equation}
If $0<s\le h_{\mu_\phi}$ then for {\rm all} $x$ we have
\begin{equation}
\label{UpperBoundIhn2} h_{\rm top}\left\{y\,:\, \a(x,y)
  \le s\right\}\le s.
  \end{equation}
\end{theorem}

\begin{proof} Let
$$
    \mathcal{A}_x(s) = \left\{y\,:\, \a(x,y)  \le s \right\}.
$$
The case $s\le h_{\mu_\phi}$ is simple. In fact, if $a>s$, we have
$$
   \mathcal{A}_x(s) \subset \limsup_{n\to \infty}
     \left\{y\,:\, \tau_n(x,y) \le 2^{an}\right\}
     = \limsup_{n\to \infty} \bigcup_{k=1}^{2^{an}} C_n(\sigma^k x).
$$
Since $C_m(\sigma^k x)\subset C_n(\sigma^k x)$ for $m >n$, we have
$$
     \mathcal{A}_x(s) \subset \bigcap_{n=1}^\infty \bigcup_{k=1}^{2^{an}} C_n(\sigma^k
  x).
$$
We have $h_{\rm top} \mathcal{A}_x(s) \le a$ since
$\{C_n(\sigma^k)\}_{1\le k\le 2^{an}}$ is a cover of  for
$\mathcal{A}_x(s)$ by $2^{an}$
cylinders of length $n$. We conclude by letting $a \downarrow s$.
Remark that $h_{\rm top} \mathcal{A}_x(s) \le s$ holds for any non
negative $s$.

We turn to the case $h_{\mu_\phi}< s \le e_{\max}$. We start with a
remark. For $\delta > 0$ and $n \ge 1$ and $0<h_1<h_2$, let
$\mathfrak{L}_n (h_1, h_2) := \mathfrak{L}_n(h_1,h_2,\delta)$ be the
set of cylinders $C$ of length $n$ such that $2^{-(h_2 - \delta)n}
\le \mu_\phi(C)\le 2^{-(h_1 + \delta)n}$. Then for $n$ sufficiently
large (depending on $h_1,h_2$ and $\delta$) we have
$$
   {\rm Card} \, \mathfrak{L}_n
(h_1, h_2) \le 2^{n  E(h_2) }
   \quad {\rm if } \ \ \  h_2< e_{\max}
$$
$$
    {\rm Card} \, \mathfrak{L}_n
(h_1, h_2) \le 2^{n  E(h_1)}
   \quad {\rm if } \ \ \  h_1 > e_{\max}.
$$
In fact, assume $h_2< e_{\max}$ (the other case may be similarly
proved). There exists a positive number $q$ such that $ E(h_2) =
P(q)+ h_2 q$. Then
$$
  2^{-q(h_2-\delta)n} {\rm Card} \, \mathfrak{L}_n
(h_1, h_2)
  \le \sum_{C \in \mathfrak{L}_n
(h_1, h_2) } \mu_\phi(C)^q
     \le 2^{n (P(q)+q\delta)}.
$$

Write
$$
     \mathcal{A}_x(s)  =  \left( \mathcal{A}_x(s) \cap \left\{ y\,:\,  \underline{h}_{\mu_\phi}(y) \le s \right\}
     \right)
      \bigcup \left( \mathcal{A}_x(s) \cap  \left\{y\,:\,  \underline{h}_{\mu_\phi}(y) > s \right\}\right).
$$
Since $h_{\rm top} \{y: \underline{h}_{\mu_\phi}(y) \le s\} \le
E(s)$, it suffices to show
\begin{equation}\label{show}
h_{\rm top} \left( \mathcal{A}_x(s) \cap \left\{y\,:\,
\underline{h}_{\mu_\phi}(y) > s \right\}\right) \le  E(s).
\end{equation}

Let
$$ \mathcal{H}(h', h'') =
 \left\{y\,:\,
 h' \le \underline{h}_{\mu_\phi}(y) \le h''
  \right\}.
$$
If all choices $s<h'<h''$  such that $h''< e_{\max}$ or $h' >
e_{\max}$ the formula
$$
h_{\rm top}  \left( \mathcal{A}_x(s) \cap \mathcal{H}(h', h''
)\right) \le  E(s)
$$
holds, then the equation (\ref{show}) also holds.

Let $s<h_1+ \delta <h'<h''<h_2 - \delta$ with $h_1$ close to $h'$
and $h_2$ close to $h''$.  Remark that $y \in \mathcal{H}(h', h'')$
implies that $C_n(y)) \in \mathcal{H}(h_1,h_2)$ for infinitely many
$n$'s. In other words
$$
    \mathcal{H}(h', h'') \subset \bigcap_{m=1}^\infty
    \bigcup_{n=m}^\infty \ \ \bigcup_{C \in \mathcal{L}_n(h_1, h_2)} C.
$$
That is to say, for any fixed $m$, $\bigcup_{n\ge m}
\mathcal{L}_n(h_1, h_2)$ is a cover of $\mathcal{H}(h', h'')$.

 Now
we construct a cover of $A_x(s) \cap \mathcal{H}(h', h'')$. For any
$s<a<h_1$, let
\begin{eqnarray*}
      \mathcal{L}_n(x; a, h_1,h_2)
      &= &\{C\in \mathfrak{L}_n
(h_1, h_2): \tau(x, C)\le 2^{an}\},\\
 N_n(x; a, h_1,h_2) & =& \mbox
{Card} \, \mathcal{L}_n(x; a, h_1,h_2).
\end{eqnarray*}
Clearly $\bigcup_{n\ge m} \mathcal{L}_n(x;a, h_1, h_2)$ is a cover
of $A_x(s) \cap \mathcal{H}(h', h'')$, because $$
  \ \  \  \   \mathcal{A}_x(s)
     \subset \bigcap_{m=1}^\infty
    \bigcup_{n=m}^\infty \ \ \bigcup_{C: \tau(x, C) \le 2^{a n}} C.
$$

Let $\gamma = h_2-a$ if $h_2\le e_{\max}$, or  $\gamma = h_1-a$ if $h_1
> e_{\max}$. Since $E'(t)<1$ when $t>h_{\mu_\phi}$, we have
   $$
       E(a+\gamma)- E(a)< \gamma, \quad \mbox{\rm i.e.}\quad
          E(a+\gamma)-\gamma <E(a).
   $$
We apply the Small Hitting Probability Lemma to $b=E(a+\gamma)$ and
$c=E(a)$ to get
$$
   \sum_n \mu_\phi\{x:  N_n(x; a, h_1,h_2,) > 2^{n E(a)}\}<\infty.
$$
By the Borel-Cantelli Lemma, for $\mu_\phi$-a.e.\ $x$, we have
$N_n(x; a, h_1,h_2) \le 2^{n E(a)}$ for $n\ge n(x)$. So,  if $m\ge
n(x)$, for any $\epsilon>0$ we have
\begin{eqnarray*}
    & & \sum_{n\ge m} \,\,\, \sum_{C\in \mathfrak{L}_n(x; a, h_1,h_2)} (\mbox{\rm diam}\, C
    )^{E(a)+\epsilon}\\
    & \le & \sum_{n\ge m} 2^{-n (E(a)+\epsilon)} \cdot 2^{n E(a)}
     \le \sum_{n\ge m} 2^{-n \epsilon} <\infty.
\end{eqnarray*}
Since $\bigcup_{n\ge m} \mathcal{L}_n(x; a, h_1,h_2)$ is a cover of
$\mathcal{A}_x(s) \cap \mathcal{H}(h', h'' )$, we have proved
$$
\dim \mathcal{A}_x(s) \cap \mathcal{H}(h', h'') \le E(a) +\epsilon.$$ We
finish the proof by letting first $\e \downarrow 0$ and
then  $a \downarrow s$.
\end{proof}

\begin{theorem}\label{cor1add1}
If $h_{\mu_\phi}<s<e_{\max}$ then for $\mu_\phi$-a.e.\ $x$ we have
$$h_{\rm top} \left\{ y\,:\,  \a(x,y)
  \le  s\right\} = E(s).$$
\end{theorem}

\begin{proof}
We simply need to prove the reverse inequality of
(\ref{UpperBoundIhn1}) in Theorem \ref{cor1add}. By multi-fractal
analysis there is a Gibbs measure with entropy $E(s)$ supported on
$\{y: h_{\mu_{\phi}}(y) = s \}$. Then Corollary \ref{fub} implies
the result.
\end{proof}

For $0<s<h_{\mu_\phi}$, the opposite inequality of
(\ref{UpperBoundIhn2}):
$$h_{\rm top}\left\{y\,:\, \a(x,y)
  \le s\right\}\ge s
  $$
also holds. But its proof is much more involved. It can not be
deduced from the mass transference principle as stated in \cite{BV}
since $\mu_\phi$ has nontrivial entropy spectrum.  In the next
section we make a substantial improvement in the mass transference
principle to multi-fractal Gibbs states.  In order to prove it, we
need to undertake a full investigation of the structure of typical
sequences.

 \setcounter{equation}{0}
\section{Typical sequences and Lower bound of $\dim_H \{y: \alpha(x, y)\le c\}$}\label{typ}

Recall that $\mu_\phi$ is a Gibbs measure associated to a normalized
H\"{o}lder potential $\phi$. A cylinder $C$ of length $n$ is said to be a
$(n,\e)$-cylinder if
$$
         2^{-(h+\e) n} \le \mu_\phi(C) \le 2^{-(h-\e) n}
$$
where  $h=h_\phi$ denotes the entropy of $\mu_\phi$. We denote by
$\mathcal{C}_{n, \e}$ the set of all $(n,\e)$-cylinders. Sometimes
we will say that a $(n,\e)$-cylinder is a good cylinder or the word
determining a $(n,\e)$-cylinder is a good word. As we shall prove,
a relatively short typical word contains
plenty of good subwords of a
fixed length and they are even different.

The following notations will be used. If $C$ and $D$ are cylinders,
we denote  by $C\star D$ the cylinder $C \cap \sigma^{-|C|} D$. If
we read $C$ and $D$ as words, $C \star D$ is nothing but the
concatenation of the words $C$ and $D$. Let $d \ge 1$ be an integer,
by $C\star_d D$ we mean $C \cap \sigma^{-(|C|+d)} D$, i.e.
$$
      C\star_d D =\bigcup_{G: |G|=d} C\star G\star D.
$$
For a set $S$, $\sharp S$ will denote the cardinality of $S$.

\subsection{Frequency of good words in a typical orbit}\
\begin{lemma} \label{GoodG}
Let $\mu_\phi$ be a Gibbs measure with  entropy
$h:=h_{\mu_\phi}>0$. For any $\e>0$, there exist an integer
$n(\e)\ge 1$ and a Borel set $\mathcal{G}_\e$ with
$\mu_\phi(\mathcal{G}_\e)>1-\e$ such that for any $x\in
\mathcal{G}_\e$ and any $n\ge n(\e)$, the cylinder $C= C_n(x)$ is a
$(n, \e)$-cylinder.
Consequently, if $n\ge n(\e)$, we have
$$
     (1-\e) 2^{(h-\e)n} \le \sharp \mathcal{C}_{n, \e} \le
     2^{(h+\e)n}.
$$
\end{lemma}

\begin{proof}
 By the Shannon McMillan Breiman theorem, for $\mu_\phi$-a.e.\ $x$ we have
 $$
     \lim_{n\to \infty } - \frac{\log \mu_{\phi}(C_n(x))}{n} = h.
 $$
Then by Egorov's theorem,
 there is a number $n(\e)\ge 1$ such that the set
\[
\G_\e:=\left\{y\in\Sigma_2\, :\, -\frac1n\log\mu_\phi(C_n(y))\in
  [h-\e,h+\e],\quad\forall n>n(\e)\right\}
\]
has measure $\mu_\phi(\G_\e)>1-\e$.

The upper estimate $\sharp \mathcal{C}_{n, \e} \le
     2^{(h+\e)n}$ follows from
$$
2^{-(h_\mu +\e)n } \sharp \mathcal{C}_{n, \e} \le \sum_{C \in
\mathcal{C}_{n, \e}} \mu_\phi(C)\le 1.
$$
The lower estimate $ (1-\e) 2^{(h-\e)n} \le \sharp \mathcal{C}_{n,
\e}$ follows from $\mathcal{G}_\e \subset \bigcup_{C\in
\mathcal{C}_{n,\e}} C$ and
$$
1-\e \le \mu_\phi(\mathcal{G_\e}) \le  \sum_{C \in \mathcal{C}_{n,
\e}} \mu_\phi(C) \le 2^{-(h_\mu -\e)n } \sharp \mathcal{C}_{n, \e}.
$$
\end{proof}

We call the set  $\mathcal{G}_\e$ the set of $\e$-good points.  By the definition of
$\mathcal{G}_\e$, we have
      $$
            \mathcal{G}_\e = \bigcap_{n=n(\e)}^\infty \bigcup_{C \in \mathcal{C}_{n,
            \e}} C.
      $$
Hence it is a $G_\delta$ set. We will write it as a decreasing limit of
open sets in the following manner
$$
    \mathcal{G}_\e = \bigcap_{N=n(\e)}^\infty \bigcap_{n=n(\e)}^N \bigcup_ {C \in \mathcal{C}_{n,
            \e}} C.
$$
This representation of $\mathcal{G}_\e$ is useful in the proof of
the following lemma.

\begin{lemma} \label{MassOfD}
Let $0<\e<1/2$ and let $L'\ge 1$ be an arbitrary
integer. For any cylinder $D$ of length $L'$, we have
$$
    \mu_{\phi}(D \cap \sigma^{-|D|} \mathcal{G}_\e) \ge
    \frac{1}{2\gamma^4} 2^{-L'\|\phi\|_\infty}
$$
where $\gamma > 1$ is the constant involved in the Gibbs property of
$\mu_\phi$ (\ref{GibbsProperty}).
\end{lemma}

\begin{proof} We first recall the following quasi-Bernoulli
property of $\mu_\phi$ (\ref{QBernoulliProperty}): for any two cylinders $A$ and $B$ we have
$$
     \mu_\phi(A \cap \sigma^{-|A|} B) \ge \frac{1}{\gamma^3}
     \mu_\phi(A) \mu_\phi(B).
$$
Let us prove the lemma. The set
$\mathcal{G}_\e$ is the decreasing limit of
 the open sets
$$
    \mathcal{G}_{N,\e} = \bigcap_{n=n(\e)}^N \bigcup_ {C \in \mathcal{C}_{n,
            \e}} C.
$$
Observe that $ \mathcal{G}_{N,\e}$ is a union of cylinders of length $N$.
Thus we have
$$
    \mu_\phi(D \cap \sigma^{-|D|} \mathcal{G}_\e) = \lim_{N\to \infty}
    \mu_\phi(D \cap \sigma^{-|D|}  \mathcal{G}_{N,\e}) =
    \lim_{N\to \infty} \sum_C \mu_\phi(D \cap \sigma^{-|D|}  C)
$$
where $C$ varies over all $N$-cylinders contained in
$\mathcal{G}_{N,\e}$.  First applying the quasi-Bernoulli
property and then using the fact that
$
     \mu_\phi(\mathcal{G}_{N,\e}) \ge 1-\e >1/2,
$ yields
$$
    \sum_C \mu_\phi(D \cap \sigma^{-|D|} C)
     \ge \frac{\mu_{\phi}(D)}{\gamma^3} \sum_C \mu_\phi(C)
     =  \frac{\mu_{\phi}(D)}{\gamma^3} \mu_\phi(\mathcal{G}_{N,\e})
     \ge \frac{\mu_{\phi}(D)}{2 \gamma^3}.
$$
To conclude, it suffices to remark that
$$\mu_\phi(D) \ge \frac{1}{\gamma} 2^{-|D| \ \|\phi\|_\infty}$$
which is assured  by the Gibbs property of $\mu_\phi$.
\end{proof}

The next theorem essentially says that a typical word of length
$2^{cL''}$ contains many good subwords of length $n$ with an
arbitrary but fixed prefix $D$ of length $L'$. We keep the notations
$n(\epsilon)$ and $\mathcal{G}_\epsilon$ appearing in
Lemma~\ref{GoodG}.

\begin{theorem}\label{TypicalSequence}
 Let $c>0$ be fixed. Let $0<\e <\min(\frac{1}{2},c)$,  $0<\eta<\frac{1}{2}$
and  $L'\ge 1$. There exist an integer $n(\e, \eta, L')\ge L' +
n(\e)$ and a Borel set $\mathcal{E}(\e, \eta, L')$ with $ \mu_\phi
(\mathcal{E}(\e,\eta, L'))>1-\eta$
 such that if  $x \in
\mathcal{E}(\e,\eta, L')$ and $L''
> n(\e, \eta, L')$,
for each $L'$-cylinder $D$ there are at least $2^{(c-\e)L''}$ points
of the finite orbit $\sigma^{j}x$  ($2^{L'}+1\le j \le 2^{cL''}$),
which fall into $D\cap \sigma^{-L'} \mathcal{G}_\e$.
\end{theorem}

\begin{proof} Let
$$
      m(L') := \frac{1}{2\gamma^4}
2^{-L'\|\phi|_\infty}
$$
be the lower bound which appeared in the last lemma. For $x\in
\Sigma_2$, define
\[
n_{D, L',\e}(x):=\inf\left\{n\in\N\, :\,
  \frac1N\sum_{j=2^{L'}+1}^{2^{L'}+N}\1_{D \cap \sigma^{-L'} \mathcal{G}_\e}(\sigma^j x)> \frac{1}{2}
  m(L'), \forall N\ge n \right\}
  \]
and
$$
   n_{L', \e}(x) = \max_D n_{D, L',\e}(x).
$$
 By Lemma~\ref{MassOfD} and  Birkhoff's
ergodic theorem we have
$$
\mu_\phi (x\in\Sigma_2\, :\, n_{L', \e}(x)<\infty) =1.$$ So, for any
$\eta>0$, there exists an integer $\widehat{n}(L', \e, \eta)$ such
that the Borel set
$$
     \mathcal{E}(L', \e,\eta):= \left\{x\in \Sigma_2: \  n_{L', \e}(x) \le \widehat{n}(L', \e,
     \eta)\right\}
$$
satisfies
$$
   \mu_{\phi}(\mathcal{E}(L', \e,\eta)) >1-\eta.
$$
Fix $n(L',\e, \eta)\ge 1$ sufficiently large so that
$$
     \frac{1}{2} m(L') [2^{\e n(L', \e, \eta)}-2^{L'}] \ge 1,
     $$
     $$
      n(L', \e,\eta) -L' \ge n(\e), \\
      $$
$$ 2^{c
     n(L',\e,\eta)} -2^{L'} \ge \widehat{n}(L', \e, \eta).
$$
Assume $x\in \mathcal{E}(L', \e,\eta)$ and $L'' \ge
n(L',\e, \eta)$. Since $N: = 2^{cL''} -2^{L'}\ge \widehat{n}(L', \e,
\eta)$,  we have
\begin{eqnarray*}
\sum_{j=2^{L'}+1}^{2^{cL''}}\1_{D \cap \mathcal{G}_\e}(\sigma^j x)
  &\ge& \frac{1}{2}m(L') [2^{cL''} -2^{L'}]\\
  &\ge& \frac{1}{2}
  m(L') [2^{\e n(L', \e,\eta)} -2^{L'}]\cdot 2^{(c-\e)L''}\\
  &\ge & 2^{(c-\e)L''}.
\end{eqnarray*}

\end{proof}

Let $C$ be a cylinder of length $n$. If $C_n(\sigma^j x)=C$, we say that the
cylinder $C$ is seen in $x$ at time $j$. Let $\e>0$, $L'<L''$ and
let $D$ be a cylinder  of length $L'$. For any $x\in \Sigma_2$, we
define a finite tree, denoted $\mathcal{T}(x,D, L',L'',\e)$,  as
follows:
\begin{itemize} \item  the nodes of
$\mathcal{T}(x,D, L',L'',\e)$ are all those cylinders $D\star G'$,
where $G'$ is a $(\ell-L',\e)$-cylinder  with $L'+n(\e)\le \ell \le
L''$, each of which  contains at least one $(L'', 2\e)$-cylinder
seen in $x$ at a moment between the
time $2^{L'}+1$ and the time $2^{cL''}$;
\item a $\ell$-cylinder $D\star G'\in\mathcal{T}(x,D, L',L'',\e)$ is the
parent of a $(\ell+1)$-cylinder $D\star G''\in\mathcal{T}(x,D,
L',L'',\e)$ if and only if  $G'' \subset G'$.
\end{itemize}
Fix $L'<L''$. For $L'+n(\e)\le \ell\le L''$, denote
$$
     T(x,D,\ell, \e) := \sharp\{D\star G' \in \mathcal{T}(x,D, L',L'',\e):
     |D\star G'|=\ell\}.
$$

Theorem~\ref{TypicalSequence} implies that if $L''$ satisfies the
condition of Theorem~\ref{TypicalSequence} and if $x \in
\mathcal{E}(L', \e,\eta)$, then in between the times $2^{L'}+1$ and
$2^{cL''}$, for each $L'$-cylinder $D$ we can see at least
$2^{(c-\e)L''}$ cylinders of length $L''$ in $x$ of the form
\begin{equation}\label{L''-good}
       D\star G'      \qquad (G'\in \mathcal{C}_{L''-L',\e}).
\end{equation}
By the quasi-Bernoulli property (\ref{QBernoulliProperty}), it is
easy to see that if $L''$ is sufficiently larger than $L'$ then the
cylinders $D\star G'$ are good in the sense
\begin{equation}\label{2eGood}
G:= D\star G' \in \mathcal{C}_{L'', 2\e}.
\end{equation}
Thus we have
$$
T(x, D, L'',\e) \ge 2^{(c- \e)L''}.
$$
Next we will prove that with big probability, for all $L'+n(\e)\le \ell\le
L''$
$$
      T(x, D, \ell,\e) \ge 2^{(c- 2\e)\ell}.
$$

\subsection{Trees associated to a typical orbit} \

Assume that $L'' \ge n(L', \e,\eta)$.  Let $L'+n(\e)\le \ell\le
L''$, $x\in \mathcal{E}(L', \e,\eta)$, and $D$ be a $L'$-cylinder.
By definition $T(x, D, \ell,\e)$ is the number of {\it different
cylinders} of the form
$$
         D \star G' \quad \mbox{\rm with}\  G' \in \mathcal{C}_{\ell-L', \e}
$$
each of which contains at least one $(L'',2\e)$-cylinder belonging
to the list $C_{L''}(\sigma^{ j} x)$, $2^{L'} +1\le j \le 2^{cL''}$.

\begin{theorem}\label{ThmTree} There exists $n_0(\e)$ such that for
sufficiently large $L''$
and for  $L'+n_0(\e) \le \ell \le L''$ we have
    $$
    \mu_\phi\left\{x \in \mathcal{E}(L',\e,\eta): T(x, D,\ell,\e) \le 2^{(c- 2\e )(\ell-L')}
    \right\}
     \le 2^{-2^{(c-2\e)L''}}.
     $$

\end{theorem}

In the rest of this subsection and the next two subsections we
prepare for the proof of this theorem, which will be presented  in
the subsection \ref{75}. We need to estimate the measures
$$
    \mu_\phi\left\{x \in \mathcal{E}(L',\e,\eta): T(x, D, \ell, \e)
    =K\right\}
$$
for  $K \le 2^{(c- 2\e )(\ell-L')}$. We will do that in the
following.

For $1\le t \le L''+ d$ (where $d:=[\omega L'']$), let
$$
   \Lambda_t =\left\{ 2^{L'} + k (L''+d) + t: 0\le k \le \frac{2^{cL''}-2^{L'}}{L''+d}
   \right\}.
$$
Fix $K$ cylinders $C_1, \cdots, C_K \in \mathcal{C}_{\ell-L',\e}$.
Let
\begin{eqnarray*}
     &&\Upsilon_t(x;C_1, C_2, \cdots, C_K) = \\
      &&\quad \quad \sharp\left\{ j \in \Lambda_t:
          C_{L''}(\sigma^j x)
         \in \mathcal{C}_{L'', 2 \e} \ \mbox{\rm implies} \
            C_{L''}(\sigma^j x) \subset D \star \widetilde{C}
          \right\}
\end{eqnarray*}
where $$ D\star \widetilde{C}: =\bigcup_{i=1}^K D\star C_i.
$$

$T(x, D, \ell, \e)= K$ means there exist $K$ different
$(\ell-L', \e)$-cylinders, say $C_1, C_2, \cdots, C_K$ such that
all
 $(L'',2\e)$-cylinders seen in $x$ in between the times $2^{L'}+1$ and $2^{cL''}$
are
contained in some of the $D\star C_i$'s, i.e.\ contained in $D\star
\widetilde{C}$. On the other hand, by Theorem~\ref{TypicalSequence},
there are at least $2^{(c-\e)L''}$ of the $(L'',2\e)$-cylinders
seen in $x$ in between the times $2^{L'}+1$ and $2^{cL''}$.
So, for at least one $t$ the number of the $(L'',2\e)$-cylinders seen at
moments belonging to $\Lambda_t$ and contained in
$D\star \widetilde{C}$ is
at least
$\displaystyle \frac{2^{(c-\e)L''}}{L''+d}$.  Thus we get
\begin{eqnarray*}
 \{x \in \mathcal{E}(L', \e,\eta): T(x, D, \ell, \e)= K\}
   \subset
      \bigcup_{t=1}^{L''+ d} \bigcup_{C_1, \cdots,C_K}
          E_t(C_1, \cdots, C_K)
\end{eqnarray*}
where the second union is taken over all possible collections $C_1,
\cdots, C_K$ of  $(\ell-L', \e)$-cylinders, and where
$$
         E_t(C_1, \cdots, C_K)=  \left\{x \in \mathcal{E}(L', \e,\eta): \Upsilon_t(x; C_1, C_2, \cdots,
         C_K) \ge \frac{2^{(c-\e)L''}}{L''+d}
           \right\}.
$$
Therefore, using the fact that the number of
$(\ell-L',\e)$-cylinders is at most $2^{(h+\e)(\ell-L')}$, we have
proved
\begin{lemma}\label{lemmaI}
\begin{eqnarray*}
   & &  \mu_\phi (x \in \mathcal{E}(L', \e,\eta): T(x, D, \ell, \e)=
    K) \nonumber \\
     & \le & (L''+d) \left(
                    \begin{array}{c}
                      2^{(h+\e)(\ell-L')} \\
                      K \\
                    \end{array}
                  \right)
                  \sup_{t; C_1,\cdots,C_K} \mu_\phi(E_t(C_1, \cdots,
                  C_K)).
\end{eqnarray*}
\end{lemma}

\subsection{Generalized quasi Bernoulli property} \

In order to estimate the measure $\mu_\phi(E_t(C_1, \cdots,
                  C_K))$, we need the following  generalized quasi Bernoulli property.



Let $A$ be any cylinder and $L\ge 1$ be any integer. For $x\in A$,
we define
    $$
       \iota_A(x) =\inf \{|A|+ k(L + d(L))  \ge 0: C_{L}(\sigma^{|A|+k(L+d(L))}x) \in \mathcal{C}_{L,
       \e}\}
    $$
where $d(L)=\lfloor \omega L\rfloor$ for some big $\omega >1$ (see
Theorem~\ref{Thm-MR}).

\begin{lemma} [Generalized quasi Bernoulli property]
\label{GQB} Let $A$ be any cylinder, $G \in \mathcal{C}_{L, \e}$ and
$\iota_A$ be defined as above. Then
   $$
    \mu_\phi(x\in A: C_L(\sigma^{\iota_A(x)}x) = G)
    \le \frac{\gamma^3}{1-2\e} \mu_\phi(A) \mu_\phi(G).
    $$
\end{lemma}

\begin{proof} Notice that
  $$\{x\in A: C_L(\sigma^{\iota_A(x)}x) = G\}
  = \bigcup_{i=0}^\infty A_i
  $$
  where
  $$ A_i= \{x\in A: C_L(\sigma^{\iota_A(x)}x) = G, \iota_A(x)
  =|A|+i (L + d)\}.$$
For $i=0$, we have
      $$
          A_0 = A \star G.
      $$
So, by the Gibbs property (\ref{GibbsProperty}) we get
$$
    \mu_\phi(A_0) \le \gamma^3 \mu_\phi(A) \mu_\phi(G).
$$
For $i\ge 1$, we have
$$
     A_i \subset
     \bigcup_{B_1, \cdots, B_i \not\in \mathcal{C}_{L,\e} } A \star B_1 \star_d \cdots \star_d B_i
     \star_d G.
$$
So, by the multi-relation (\ref{MR}) we get
$$
   \mu_{\phi}(A_i) \le \gamma^3 (1+\beta^d)^i \mu_\phi(A) \mu_\phi(G)
   \left(\sum_{B \not\in \mathcal{C}_{L,\e}} \mu_{\phi}(B)\right)^{\hspace{-0.1cm}i}.
$$
Since $\sum_{B \not\in \mathcal{C}_{L,\e}} \mu_{\phi}(B) \le
\mu_{\phi}(\mathcal{G}_{L, \e}^c)\le \e$, we get
$$
   \mu_{\phi}(A_i) \le \gamma^3(\e (1+\beta^d))^i \mu_\phi(A) \mu_\phi(G).
$$
Thus
\begin{eqnarray*}
   \mu_\phi(x\in A: C_L(\sigma^{\iota_A(x)}x) = G)
   & \le & \gamma^3 \mu_\phi(A) \mu_\phi(G) \sum_{i=0}^\infty (\e
   (1+\beta^d))^i\\
   &=& \frac{\gamma^3}{1-\e(1+\beta^d)} \mu_\phi(A) \mu_\phi(G).
\end{eqnarray*}
We finish the proof by observing that $\beta<1.$
\end{proof}

\subsection{Estimation of $\mu_\phi(E_t(C_1, \cdots,
                  C_K) )$}\

Let $t$  be fixed.
We define inductively
\begin{eqnarray*}
  \iota_1(x) & = &  \inf\{ j \in \Lambda_t: C_{L''}(\sigma^j x)\in \mathcal{C}_{L'', 2\e} \};\\
  \iota_{k+1}(x) & = &  \inf\{ j \in \Lambda_t: j> \iota_k(x); C_{L''}(\sigma^j x) \in \mathcal{C}_{L'', 2\e}
   \}.
\end{eqnarray*}
Let
\begin{equation}\label{tildeN}
     \widetilde{n}:= \frac{2^{(c-\e)L''}}{L''+d}.
\end{equation}
We have
$$
      \iota_i(x) <\infty \quad \mbox{\rm if} \ x \in E_t(C_1,\cdots,
      C_K), \mbox{and \ if} \  i \le
      \widetilde{n}.
$$
Then
\begin{equation}
\mu_\phi(E_t(C_1, \cdots,C_K)) \le \sum \mu_\phi\left(x:
\sigma^{\iota_i(x)} x \in F_i,   1\le \forall i\le \widetilde{n}
\right)
\end{equation}
where the sum is taken over all $F_i$'s with the property
$$
       F_i \in \mathcal{C}_{L'', 2 \e}, \quad F_i \subset
       D \star \widetilde{C}\quad  (1\le \forall i \le \widetilde{n}).
$$

\begin{lemma} Let $n\ge 1$ and let $F_i \in \mathcal{C}_{L'', 2\e}$ with $1\le i\le
n$. We have
\begin{eqnarray*}
     \mu_\phi\left(x: C_{L''}(\sigma^{\iota_i(x)} x) = F_i; i=1,2, \cdots, n\right)
\le  \left(\frac{\gamma^3}{1-4\e}\right)^{\vspace{-0.3cm}n}
        \prod_{i=1}^n
        \mu_\phi (F_i).
\end{eqnarray*}
\end{lemma}

\begin{proof} We prove it by induction on $n$. Let
$$
       \mathcal{Q}_n = \{x: C_{L''}(\sigma^{\iota_i(x)} x) = F_i; i=1,2, \cdots,
       n\}.
$$
Write
$$
   \mathcal{Q}_{n+1}
      = \mathcal{Q}_n \cap \{x: C_{L''}(\sigma^{\iota_{n+1}(x)} x) =
      F_{n+1}\}.
$$
Notice that $\mathcal{Q}_n$ is a disjoint union of cylinders, say
$$
       \mathcal{Q}_n = \bigcup A_j.
$$
Furthermore if $x \in A_j$ we have
   $$
        C_{L''}(\sigma^{\iota_{n+1}(x)}x)= F_{n+1}
        \Longleftrightarrow C_{L''} (\sigma^{\iota_{A_j}(x)} x)=F_{n+1}.
   $$
Thus, using the generalized Bernoulli property (Lemma~\ref{GQB}), we
have
\begin{eqnarray*}
   \mu_\phi(\mathcal{Q}_{n+1})
   & = &  \sum_j  \mu_\phi (x\in A_j,
      C_{L''}(\sigma^{\iota_{A_j}(x)} x)=
   F_{n+1})\\
    & \le  &  \frac{\gamma^3}{1-4\e} \sum_j  \mu_\phi (A_j)
      \mu_\phi(
   F_{n+1})\\
    & = & \frac{\gamma^3}{1-4\e}  \mu_\phi (\mathcal{Q}_n)
      \mu_\phi(
   F_{n+1}).
\end{eqnarray*}
\end{proof}

\begin{lemma}\label{lemmaII}
   $$
   \mu_\phi(E_t(C_1, \cdots, C_K))
   \le \left(2 \gamma^6
         K
         2^{(-h+\e)(\ell-L')}\right)^{\frac{2^{(c-\e)L''}}{L''+d}}.
   $$
\end{lemma}

\begin{proof} By the last lemma, we have
$$
\mu_\phi(E_t(C_1, \cdots, C_K)) \le \left( \frac{\gamma^3}{1-4\e}
\right)^{\vspace{-0.3cm}\widetilde{n}}
        \sum_{F_1,\cdots, F_{\widetilde{n}}} \prod_{i=1}^{\widetilde{n}}
        \mu_\phi (F_i)
   $$
where the sum is taken over all collections $F_1, \dots, F_n$'
consisting of different $(L'', 2\e)$-cylinder contained in $D
\star\widetilde{C}$. Recall that $\tilde{n}$ is defined in
(\ref{tildeN}).

Since $\mu_\phi (D\star C_i)\le \gamma^3 \mu_\phi(D) \mu_\phi(C_i)$
and  $\mu_\phi(C_i) \le 2^{(-h+\e)(\ell-L')}$, we have
$$
    \sum_{F \in \mathcal{C}_{L'', 2\e}, F\subset D \star
    \widetilde{C}}\mu_\phi(F)
    \le \mu_\phi(D\star \widetilde{C}) \le K \gamma^3 2^{(-h+\e)(\ell-L')}.
$$
So,
$$
\mu_\phi(E_t(C_1, \cdots, C_K)) \le \left(\frac{\gamma^6}{1-4\e}
         K 2^{(-h+\e)(\ell-L')}\right)^{\vspace{-0.3cm}\widetilde{n}}.
   $$
\end{proof}

\subsection{Number of branches of a tree: Proof of Theorem~\ref{ThmTree}}
\label{75}\

By Lemmas~\ref{lemmaI} and \ref{lemmaII}, we have
\begin{eqnarray}\label{Thm7-4-1}
& & \hspace{-4em} \mu_\phi \Big (x\in \mathcal{E}(L',\e,\eta ):
     T(x, D, \ell, \e)=K \Big )     \nonumber \\
&\le &
 (L''+d) \left(
                    \begin{array}{c}
                      2^{(h+\e)(\ell-L')} \\
                      K \\
                    \end{array}
                  \right)
                  \left(2\gamma^6
         K
         2^{-(h-\e)(\ell-L')}\right)^{\frac{2^{(c-\e)L''}}{L''+d}}.
\end{eqnarray}
For $K\le 2^{(c-2\e )(\ell-L')}$ and  for  $\ell \le L''$, we have
on one hand
\begin{equation}\label{Thm7-4-2}
   \left(
                    \begin{array}{c}
                      2^{(h+\e)(\ell-L')} \\
                      K \\
                    \end{array}
                  \right)
                  \le 2^{(h+\e)(\ell-L') K}
                  \le 2^{(h+\e)L'' 2^{(c-2\e)L''}};
\end{equation}
and on the other hand
$$
     K 2^{-(h-\e)(\ell-L')} \le 2^{(c-h-\e )(\ell-L')},
$$
which implies that there exists an integer $n_0(\e)$ such that if
$\ell-L'\ge n_0(\e)$ we have
\begin{equation}\label{Thm7-4-3}
       2\gamma^6 K 2^{-(h-\e)(\ell-L')} \le \frac{1}{2},
       \quad \text{i.e.}\ \
         2\gamma^6 2^{-(h-c-\e)(\ell-L')} \le \frac{1}{2}.
\end{equation}
So, from (\ref{Thm7-4-1}), (\ref{Thm7-4-2}) and (\ref{Thm7-4-3}) we
get
\begin{eqnarray}\label{Thm7-4-4}
 & & \hspace{-4em} \mu_\phi \Big ( x\in
\mathcal{E}(L',\e,\eta):
      T(x, D, \ell, \e)=K \Big )    \nonumber  \\
     & \le&  (L''+d)\cdot 2^{(h+\e)L'' 2^{(c-2\e)L''} -
     \frac{2^{(c-\e)L''}}{L''+d}}.
\end{eqnarray}
Choose $L''$ sufficiently large so that
\begin{equation}\label{Thm7-4-5}
      (h+\e)L'' 2^{(c-2\e)L''} \le \frac{1}{2}\cdot
      \frac{2^{(c-\e)L''}}{L''+d}.
\end{equation}
>From (\ref{Thm7-4-4}) and (\ref{Thm7-4-5}), we get
$$
  \mu_\phi(x\in
\mathcal{E}(L',\e,\eta):
     T(x, D, \ell, \e)=K)
     \le (L''+d) \cdot 2^{-
     \frac{2^{(c-\e)L''}}{2(L''+d)}}.
$$
Summing  over all $K\le 2^{(c-2\e)(\ell-L')}$, we obtain
\begin{eqnarray*}
  & & \hspace{-4em} \mu_\phi\left\{ x\in
\mathcal{E}(L',\e,\eta):
     T(x, D, \ell, \e)\le 2^{(c-2\e)(\ell-L')} \right\}\\
    & \le &  (L''+d) \cdot 2^{(c-2\e)(\ell-L')} \cdot 2^{-
     \frac{2^{(c-\e)L''}}{2(L''+d)}}
     \le 2^{-2^{(c-2\e)L''}}
\end{eqnarray*}
for large $L''$, because $2^{-\frac{2^{(c-\e)L''}}{2(L''+d)}}$ tends
to zero superexponentially fast.

\subsection{The Cantor set and lower bound of $\dim_H \{y: \alpha(x, y)\le c\}$}\

The next theorem is an improvement of the mass transference principle
\cite{BV} to the multi-fractal measure $\mu_\phi$.

\begin{theorem}\label{revers} {\rm (Multi-fractal mass transference
    principle)} For $0<c<~h_{\mu_\phi}$, and for
  $\mu_{\phi}$-a.e.\ $x$ we have
 $$h_{\rm top}\left\{y\,:\, \a(x,y)
  \le c \right\}\ge c.$$
\end{theorem}

\begin{proof}
Let $\e>0$  be an arbitrary  small number. We can find an increasing
sequence of integers $(L_k)_{k\ge 0}$ such that
\begin{equation}\label{ChoiceOfLk}
    L_0=0 , \quad  2^{-2^{(c-2\e)L_{k}}} \le \frac{\e}{2^{k+2}}.
\end{equation}
and
that for each $k\ge 1$, the
couple $(L',L'')=(L_{k-1}, L_{k})$ satisfies the condition of
Theorem~\ref{ThmTree}. Apply Theorem~\ref{ThmTree} to
$L'=L_{k-1}, L''=L_{k}$ and $\eta= \frac{\e}{2^{k+1}}$ to get
$\mathcal{E}_k(\e) := \mathcal{E}(L',\e,\eta)$. It has
the properties that
 \begin{equation}\label{MeasureOfEk}
    \mu_\phi(\mathcal{E}_k(\e)) > 1- \frac{\e}{2^{k+1}};
\end{equation}
and that there is a subset $\mathcal{E}^*_k(\e)$ of
$\mathcal{E}_k(\e)$ with
\begin{equation}\label{MeasureOfE*k}
\mu_\phi(\mathcal{E}_k(\e) \setminus \mathcal{E}^*_k(\e)) <
\frac{\e}{2^{k+1}}
\end{equation} such that for any $x\in \mathcal{E}^*_k(\e)$,
any $L_{k-1}$-cylinder $D$
 and any $L_{k-1}+n_0(\e) \le \ell \le  L_{k}$ we
have
$$
    T(x, D, \ell, \e) \ge 2^{(c-2\e)(\ell-L_{k-1})}.
$$
Define
$$
        \mathcal{E}^*(\e) = \bigcap_{k=1}^\infty \mathcal{E}^*(L_k,
        \e).
$$
Equations (\ref{MeasureOfEk}) and (\ref{MeasureOfE*k}) imply that
$\mu_\phi (\mathcal{E}^*_k(\e)))\ge 1- \frac{\e}{2^k}$ and
\begin{equation}\label{MeasureOfE}
     \mu_\phi(\mathcal{E}^*(\e)) \ge 1 - \sum_{k=1}^\infty
     \frac{\e}{2^k} = 1- \e.
\end{equation}
For $x\in \mathcal{E}^*(\e)$, we have
\begin{equation}\label{T-ell}
    T(x, D, \ell, \e) \ge 2^{(c-2\e)(\ell-L_{k-1})}
\end{equation}
for all $L_{k-1}$-cylinders $D$  and all $L_{k-1}+n_0(\e) \le \ell
\le L_{k}$.

Now, for each $x\in \mathcal{E}^*(\e)$, we construct a Cantor
set as follows.

{\it First step}:  for $n_0(\e) \le \ell\le L_1$,  consider the
family $\mathfrak{C}_\ell(x)$ of
$(\ell,\e)$-cylinders which contain at least one $(L_1,
2\e)$-cylinder seen in $x$ between the times $1$ and $2^{c
L_1}$. This yields a tree $\mathfrak{T}_{L_1}(x)$
of height $L_1$. The nodes of the tree $\mathfrak{T}_{L_1}(x)$ are
the $(\ell,\e)$-cylinders, with $n_0(\e)\le \ell \le L_1$,
belonging to $\mathfrak{C}_\ell(x)$.
The edges are defined by the containment relation.
We will extend this tree inductively.

{\it Second step}: Let $k\ge 2$. Suppose that we have constructed a
tree $\mathfrak{T}_{L_{k-1}}(x)$ of height  $L_{k-1}$. We will
construct a tree of height $L_{k}$.  Let
$$
L' =L_{k-1}, \quad L'' =L_{k}.
$$
Fix a $L'$-cylinder $D$ seen in $x$ before time $2^{c L'}$, which is
the label of a node of the tree $\mathfrak{T}_{L_{k-1}}(x)$ at level
$L_{k-1}$. For $L'+ n_0(\e) \le \ell \le L''$, take all
$(\ell,\e)$-cylinders that contain at least one  $(L'',
2\e)$-cylinder of the form $D\star G$ seen in $x$ between the times
$2^{L'}+1$ and $2^{c L''}$. As before we denote this family by
$\mathfrak{C}_\ell(x)$ (both $D$ and $G$ varying). The tree
$\mathfrak{T}_{L_{k}}(x)$ is obtained from
$\mathfrak{T}_{L_{k-1}}(x)$ by adding branches to each $D$. That is
to say, by splitting $D$ into $(\ell,\e)$-cylinders belonging to
$\mathfrak{C}_\ell(x)$.

We define
$$
    C_\infty(x) = \bigcap_{k=1}^\infty\
    \bigcap_{\ell=L_{k-1}+n_0(\e)}^{L_k} \ \bigcup_{C \in
    \mathfrak{C}_\ell(x)} C.
$$
We have $C_\infty(x) \subset \{y: \alpha(x, y)\le c\}$, since for
any $y \in C_\infty(x)$ and for all $k\ge 1$
$$
    y \in \bigcup_{C \in \mathfrak{C}_{L_k}(x)} C,
$$
i.e.\ $y \in C_{L_k}(\sigma^j x)$ for some $$
2^{L_{k-1}}+1\le j\le 2^{cL_k}.$$

We claim that $\dim_H C_\infty(x) \ge c- 2\e$. In fact, for
$L_{k-1}+ n_0(\e) \le \ell\le L_k$, we have
\begin{eqnarray*}
    \log_ 2 \sharp \mathfrak{C}_\ell(x)
     & \ge &  (c-2\e)(\ell- L_{k-1}) +
    \sum_{j=1}^{k-1}(c-2\e)(L_j-L_{j-1})\\
    & \ge & (c-2\e)\ell
\end{eqnarray*}

Define a probability measure $\nu$ on $C_\infty(x)$ by
$$
      \nu(C) = \frac{1}{\sharp \mathfrak{C}_\ell(x)} \qquad
          (\forall C \in \mathfrak{C}_\ell(x) \  \mbox{\rm and} \ l \in \N).
$$
It is clear that (note $n(\e)$ does not depend on $L_k$)
$$
     \nu(C) \le  2^{-(c-2\e)\ell}.
$$

Thus we have proved that with probability bigger than $1-\e$ we have
$$
        \dim_H \{y: \alpha(x, y) \le c\} \ge c-2\e.
$$
\end{proof}

Remark:  The proofs in this section can be used to obtain
a more precise estimate on the growth rate of the tree,
however this estimate is not necessary for our purpose.
Namely one can show that $L_{k-1} \ll l \le \frac{c}{h}L_k$
then
$$
T(x,D,l,\e) \ge 2^{(h-3\e)l}.
$$
This implies that the upper box counting dimension of the
corresponding Cantor set is $h-3\e$ while the lower box
dimension equals the Hausdorff dimension equals $c-2\e$.

\setcounter{equation}{0}
\section{Results for the full shift}\label{sec8}

Our strategy is to prove all the theorems in the symbolic framework
and then transfer them to the circle. Let us get together the
already obtained results in the symbolic framework.

\begin{lemma}\label{lemma3}

For  $0 < \kappa < \infty$ we have $\mu_\phi$-a.e.
$$\sup \{E(t): \frac1t \le \kappa\} \ge \dim_H \Fhn \ge  \sup \{E(t): \frac1t < \kappa\}.$$
For  $\kappa \le 1/h_{\mu_\phi}$ (i.e. $1/\kappa \ge h_{\mu_\phi}$)
we have $\mu_\phi$-a.e.
$$\sup \{ E(t): \frac1t \ge \kappa\} \ge \dim_H \Ihn \ge  \sup \{
E(t): \frac1t > \kappa\},$$\\
and for $\kappa > 1/h_{\mu_\phi}$ (i.e. $1/\kappa < h_{\mu_\phi}$)
we have $\mu_\phi$-a.e.
$$\dim_H \Ihn = 1/\kappa.$$
\end{lemma}

\begin{proof}

The first line is a consequence of Lemma~\ref{lemma2},
Theorem~\ref{5.3} and Theorem~\ref{var}.

The second line is a consequence of Lemma~\ref{lemma2},
Theorem~\ref{cor1add} 
and Theorem~\ref{var}.

The third line is a direct consequence of Lemma~\ref{lemma2},
Theorems \ref{cor1add} and \ref{revers}.
\end{proof}

\begin{corollary}\label{var2}
Let $1/\kappa \in (e^-,e^+)$. Then for $\mu_\phi$ a.e.\  $x$
\[
\dim_H\Fhn =\max_{\nu - ergodic}\{h_\nu\, : \a (x,y)\le\frac1{\kappa}\,\,\nu -a.e. y\}.
\]
For $1/\kappa \in (h_{\mu_\phi},e^+)$ and $\mu_\phi$ a.e.\  $x$
\[
\dim_H\Ihn =\max_{\nu - ergodic}\{h_\nu\, : \a (x,y)\ge\frac1{\kappa}\,\,\nu -a.e. y\}.
\]
\end{corollary}

The properties of the entropy spectrum which were stated in the background
section immediately imply the following corollary.
\begin{corollary}\label{theoremMFA}
For $1/\kappa \in (e^-,e^+)$ and $\mu_\phi$ a.e.\  $x$  we have

$$\sup_{-P'(q) \ge \frac{1}{\kappa}}\left[ P(q\phi) - P'(q\phi)q\right ]  \ge
dim_H \Fhn \ge \sup_{-P'(q) > \frac{1}{\kappa}}\left[ P(q\phi) -
P'(q\phi)q\right ].$$ For $1/\kappa \in (h_{\mu_\phi},e^+)$ and
$\mu_\phi$ a.e.\  $x$ we have
$$\sup_{-P'(q) \le \frac{1}{\kappa}}\left[ P(q\phi) - P'(q\phi)q\right ]  \ge
dim_H \Ihn \ge  \sup_{-P'(q) < \frac{1}{\kappa}}\left[ P(q\phi) - P'(q\phi)q\right ].$$
\end{corollary}

If we consider a typical potential, then the function $E(t)$ is
continuous on the nontrivial interval $(e^-,e^+)$, equals 0 on the
endpoints (see \cite{S'}). Hence the right hand side and left hand
side inequalities in Lemma~\ref{lemma3} and
Corollary~\ref{theoremMFA} are equal.  Since the maximum value of
$E(t)$ is attained at the value $t = - \int_{\Sigma_2^+} \phi \,
d\mm$ and equals $h_{\rm top}(\Sigma_2^+) = 1$ we have the following
corollary.

\begin{corollary} For a typical potential and $\mu_\phi$ a.e.\ $x$ we have
\begin{eqnarray*}
&\dim_H \Fhn = h_{\rm top}(\Sigma_2^+) = 1  &\hbox{for } \kappa \ge \frac{1}{- \int \phi \, d\mm.},\\
&\dim_H \Ihn =  h_{\rm top}(\Sigma_2^+) = 1  & \hbox{for } \kappa \le \frac{1}{- \int \phi \, d\mm.}.
\end{eqnarray*}
Let $q_{\kappa}$ be the number such that $P'(q_{\kappa} \phi) = - \frac1{\kappa}.$ Then
\begin{eqnarray*}
&\dim_H \Fhn = E\left (\frac{1}{\kappa} \right )
= P\left (q_{\kappa} \phi \right ) + \frac{1}{\kappa} q_{\kappa}
\quad &\hbox{for } \kappa < \frac{1}{- \int \phi \, d\mm.},\\
&\dim_H \Ihn = E\left (\frac{1}{\kappa} \right ) = P\left (q_{\kappa} \phi \right )
+ \frac{1}{\kappa}t q_{\kappa} \quad &\hbox{for }
\frac{1}{h_{\mu_\phi}} \ge  \kappa > \frac{1}{- \int \phi \, d\mm.}.
\end{eqnarray*}
\end{corollary}

Finally we come to the answer of the symbolic version of question
({\bf Q2}).
\begin{lemma}
For $\mu_\phi$ a.e.\ $x$ we have
$$
\Fhn = \emptyset  \hbox{ for }  \kappa < \frac1{e^+} = \frac{1}{
\max_{\mu \ \rm{ ergodic}} \int (-\phi) \, d\mu}=
\kappa^F_{\phi,\Sigma_2^+}.
$$
\end{lemma}

\begin{proof}
>From multi-fractal analysis, it is well known that
$$e^+ = \max_{\nu} \int (-\phi) \, d\nu = \max_{ y \in \Sigma_2^+} h_{\mu_\phi}(y).$$
Therefore $$ \Fhn \subset \{y: \a(x,y) \ge 1/\kappa \text{ and }
\hb_{\mu_\phi}(y) \le e^+ < 1/\kappa\}= \emptyset$$
by Lemma~\ref{lemma2}, Lemma~\ref{OW} and Theorem~\ref{thm5.2}.
\end{proof}

Using the techniques developed in the previous sections we can
conclude a strong theorem on the structure of typical sequences.
The subword structure of a typical sequence up to time $L$ is completely
determined by the entropy spectrum of the measure.

\begin{corollary}\label{structure}
Consider $n \ll L$ sufficiently large, a typical point $x$ and
the set of cylinders $C_n$ of length $n$  satisfying $\mu(C_n) \sim
2^{-\beta n}$ which are subwords of the cylinder $C_L(x)$, i.e.\
the orbit of $x$ hits the cylinder $C_n$ before time $L$.
Then
$$\sharp \big( C_n \big )  \sim  \max(0, 2^{\min (E(\beta),E(\beta) - \beta + (\log
  L)/n)n}).$$
\end{corollary}
Here $a_n \sim b_n$ means that the ratio $a/b$ is subexponential
in $n$.


\setcounter{equation}{0}
\section{Extensions to subshifts of finite type}\label{sec9}

The previous results can be extended in a canonical way to subshifts
of finite type: $\Sigma_2^+$ is replaced by a subshift space
$\Sigma_A$ and $\mu_{\phi}$ and $\mu_\psi$ by two Gibbs measures of
the subsystem $\sigma: \Sigma_A \to \Sigma_A$. Extensions to
symbolic spaces of several symbols are also obvious.

Here we consider another kind of extension.  Given a compact subset
$K$ in $\Sigma_2^+$. What can we say about $K\cap \Ihn$ and $K\cap \Fhn$ ? We
assume that the reference measures $\mu_\phi$ and $\mu_\psi$ are
Gibbs measure of the {\em full shift} $\sigma: \Sigma_2^+\to
\Sigma_2^+$. We can answer this question when $K=\Sigma_A$ is a
subshift of finite type.
  The
proofs are still slight modifications of those for the full shift,
thus we only sketch them briefly here.  We will emphasize the differences.

Let $\Sigma_A\subset \{0,1\}^{\nn}$ be a subshift of finite type. We
are interested in the following two sets:
\def\FA{\mathcal{F}^{\kappa}_A(x)}
\def\IA{\mathcal{I}^{\kappa}_A(x)}
$$\FA := \Fhn \cap \Sigma_A \quad \text{and} \quad \IA := \Ihn \cap \Sigma_A.$$
Recall that $\mu_\phi(\Sigma_A)=0$ if $\Sigma_A\neq\{0,1\}^\nn$
because $\Sigma_A$ is a closed invariant set ($\sigma \Sigma_A \subset \Sigma_A$)
and $\mu_{\phi}$ is of full support and ergodic.

The analysis of these sets is related to the determination of the
following restricted entropy spectrum: Recall that $-\int \phi
d\mu_{\psi}$ is nothing but the conditional entropy of $\mu_{\phi}$
relative to $\mu_{\psi}$. Let
$$
E_A(\alpha):=\dim_H\{y\in \Sigma_A:\ h_{\mu_\phi}(y)=\alpha\}.
$$
We list some facts concerning $E_A(\alpha)$ which are needed to
modify the proofs.
\begin{enumerate}
\item{} Clearly the restriction $\phi|_{\Sigma_A}$ is a H\"{o}lder
  function.
\item Let ${P_A}(\psi)$
be the pressure of a potential $\psi:\ \Sigma_A\to \rr$ related to
the subsystem $\sigma: \Sigma_A\to \Sigma_A $. Then
$${P_A}(\phi|_{\Sigma_A}) \le 0.$$
This a consequence of the variational  principle:
\begin{eqnarray*}
{P_A}(\phi|_{\Sigma_A}) & = &\max_{\mu \ \text{inv on} \
\Sigma_A} (h_\mu + \int_{\Sigma_A} \phi \, d\mu)  \\
      & \le & \max_{\mu \
\text{inv on} \ \Sigma} (h_\mu + \int_{\Sigma} \phi \, d\mu) 
 = P(\phi)=0.
 \end{eqnarray*}
\item{} $\phi_A(x):=\phi|_{\Sigma_A}-{P_A}(\phi|_{\Sigma_A})$
is normalized in the sense that ${P_A}(\phi_A)=0$.
\item{} Let $\mu_{\phi_A}$ be the Gibbs measure on $\Sigma_A$
associated to $\phi_A$. It is related to the original Gibbs measure
$\mu_\phi$ by
$$
\mu_{\phi_A}(C_n(x))\approx
e^{S_n\phi_A(x)}=e^{S_n\phi(x)-n{P_A}(\phi|_{\Sigma_A})}\approx
e^{-n{P_A(\phi|_{\Sigma_A}}(\phi)}\mu_\phi(C_n(x))
$$
for $x \in \Sigma_A$.  Here $\approx$ means that the ratio is bounded
between two constants independent of $n$.
\item{} Consequently, if one of the local entropies
$h_{\mu_{\phi_A}}(x)$ or
$h_{\mu_\phi}(x)$ is well defined then both are well defined and we have
$$
h_{\mu_{\phi_A}}(x)=h_{\mu_\phi}(x)+{P_A}(\phi|_{\Sigma_A}), \ x
\in\Sigma_A.
$$
\item{} The following spectrum is well known from multi-fractal analysis
$$ \aligned
\widetilde{E}_A(\beta):=\dim_H\{y\in \Sigma_A: \
h_{\mu_{\phi_A}}(y)=\beta\}.
\endaligned
$$
The condition $h_{\mu_{\phi_A}}(y)=\beta$ is equivalent to
 $\lim_{n\to \infty } n^{-1}(S_n(-\phi_A)(y))=\beta$.
\end{enumerate}

Now, by (5) and (6), we get that the spectrum $E_A(\cdot)$ is
expressed in term of the known spectrum $\widetilde{E}_A(\cdot)$:
$$
\aligned
E_A(\alpha)&=\dim_H\{y\in \Sigma_A:\ h_{\mu_\phi}(y)=\alpha\}\\
&=\dim_H\{y\in \Sigma_A:\ h_{\mu_{\phi_A}}(y)=\alpha+{P_A}(\phi)\}\\
&=\widetilde{E}_A(\alpha+{P_A}(\phi)).
\endaligned
 $$
Furthermore, the set $\{y\in \Sigma_A:\ h_{\mu_\phi}(y)=\alpha \}$
is empty, so $E_A(\alpha)=0$ unless
 \bequ\label{star}
\tilde{e}_A^-\le\alpha+{P_A}(\phi)\le \tilde{e}_A^+ \nequ where
$\tilde{e}_A^+,\tilde{e}_A^-$ are respectively the maximal and
minimal entropy of $h_{\mu_{\phi_A}}$.  That is
$$
\aligned \tilde{e}_A^+&=\sup_{  {\rm supp} \mu\subset
\Sigma_A}\int(-\phi_A)d\mu=\sup_{  {\rm supp} \mu\subset
\Sigma_A}\int(-\phi)d\mu+{P_A}(\phi|_{\sigma_A})\\
\tilde{e}_A^-&=\inf_{  {\rm supp}
\mu\subset\Sigma_A}\int(-\phi_A)d\mu=\inf_{  {\rm supp} \mu\subset
\Sigma_A}\int(-\phi)d\mu+{P_A}(\phi|_{\Sigma_A}).
\endaligned
$$
Define
$$
   e_A^{-}:=\inf_{  {\rm supp} \mu\subset
\Sigma_A}\int(-\phi)d\mu,\qquad e_A^{+} :=\sup_{  {\rm supp}
\mu\subset \Sigma_A}\int(-\phi)d\mu.
$$
 So, (\ref{star}) is equivalent to
$$
e_A^{-}\le\alpha\le e_A^{+}.
$$
Thus $E_A(\alpha)\le E(\alpha)$ because
$$
\aligned \widetilde{E}_A(\alpha+{P_A}(\phi))&=\sup_{\substack{
  {\rm supp}\mu\subset\Sigma_A\\\int(-\phi_A)d\mu=\alpha+{P_A}(\phi)}}h_\mu
=\sup_{\substack{  {\rm
supp}\mu\subset\Sigma_A\\\int(-\phi)d\mu=\alpha}}h_\mu \\ & \le
\sup_{\int(-\phi)d\nu=\alpha}h_\nu=E(\alpha).
\endaligned
$$

Let $e_A^{\max}$ be the unique value for which $E_A(\alpha)$ attains
its maximum (supported by the Parry measure). In particular
$E_A(e_A^{\max}) = \dim_H(\Sigma_A)$. Then we can conclude

\begin{theorem}

$$
      \dim_H \FA   =       \begin{cases}
                           \dim_H(\Sigma_A) & \mbox{\rm if} \ \ \frac{1}{\kappa}\le e_A^{\max}\\
                           E_A(\frac{1}{\kappa}) & \mbox{\rm if} \ \ \frac{1}{\kappa}> e_A^{\max}
                          \end{cases}
$$
$$
\text{and} \quad \FA = \emptyset \quad \text{if} \quad  \frac{1}{\kappa}> e_A^+.$$
\end{theorem}

\begin{theorem}
$$
      \dim_H \IA  =      \begin{cases}
                           \frac{1}{\kappa} + P_A(\phi|_{\Sigma_A})
                           & \mbox{\rm if} \ \
                           -P_A(\phi|_{\Sigma_A}) \le
                           \frac{1}{\kappa}\le
                           h_{\mu_{\phi_A}}- P_A(\phi|_{\Sigma_A})\\
                           E_A(\frac{1}{\kappa}) & \mbox{\rm if} \ \
                           h_{\mu_{\phi_A}}- P_A(\phi|_{\Sigma_A}) \le
                           \frac{1}{\kappa} \le
                           e_A^{\max}\\
                           \dim_H(\Sigma_A) & \mbox{\rm if} \ \
                           \frac{1}{\kappa}
                           \ge  e_A^{\max}\\
                          \end{cases}
$$
$$ \text{and} \quad \IA = \emptyset \quad \text{if} \quad  \frac{1}{\kappa}
                           <  -P_A(\phi|_{\Sigma_A}).$$
\end{theorem}
\noindent Remark:  Unlike the full shift case $\IA$ is empty for
large $\kappa$.
\begin{proof}
The only statement in the two theorems which differs from the full
shift is that $\IA$ may be empty.   Fix $\e > 0$.  Let
$\frac1{\kappa} < -P_A(\phi|_{\Sigma_A}) - \e.$ Then by (\ref{star})
we have
$$h_{\mu_{\phi}}(y)
  \ge \tilde{e}^-_A -P_A(\phi|_{\Sigma_A})$$
for all $y \in \Sigma_A$.
Then, by Lemma~{\ref{lemma2}}
\begin{eqnarray*}
\IA  & \subset & \left\{y \in \Sigma_A: \  \alpha(x,y) < \frac1{\kappa} +
  \epsilon\right\}\\
& = &
 \{y \in \Sigma_A: \ \alpha(x,y) <
\frac{1}{\kappa} + \epsilon, \
h_{\mu_{\phi}}(y) \ge \tilde{e}^-_A -  P_A(\phi|_{\Sigma_A}) \}
\\
& \subset & \{y \in \Sigma_A: \ \alpha(x,y) <
-P_A(\phi|_{\Sigma_A}), \
h_{\mu_{\phi}}(y) \ge \tilde{e}^-_A -  P_A(\phi|_{\Sigma_A}) \}
\\
& \subset & \bigcup_{j=0}^{\infty} \{ y \in \Sigma_A: \
h_{\mu_{\phi}}(y) \in \big [j \e,(j+1)\e \big ) + \tilde{e}^-_A - P_A(\phi|_{\Sigma_A}) \}.
\end{eqnarray*}
Thus,  Lemma \ref{lemma1add} with $K = 2^{\frac1{\kappa}n}, L = \max
(2^{E_A(\tilde{e}^-_A + j\e )},2^{(\tilde{e}^-_A + (j+1)\e})$ and
$N=1$ implies that each of the (countably many) sets on the right
hand side is empty for $\mu_{\phi}$-a.e.\ $x$.
\end{proof}

\setcounter{equation}{0}
\section{Transferring to the circle}\label{sec10}
In this section we show that  the results of the section \ref{sec8}
hold for the doubling map of the circle, i.e.\ replacing $\Fhn,\Ihn$
by $\Fn,\In$. Recall that the projection $\pi: \Sigma \to \S$ was
defined in the section \ref{back}. For $y\in\Sigma_2$; $y\ne
1^\infty ,0^\infty$ let
\[
C_n^*(y):=C_n^-(y)\cup C_n(y)\cup C_n^+(y)
\]
where $C_n^-(y)$ denotes the cylinder of length $n$ preceding $C_n(y)$ in the
lexicographical order and $C_n^+(y)$ denotes the immediate successor.

\begin{theorem} For $\mu_\phi$ a.e.\ $x$ we have
$$\dim_H (F^\kappa \pi(x)) = \dim_H \pi (\Fhn)$$
$$\dim_H (I^\kappa \pi(x)) = \dim_H \pi (\Ihn).$$
\end{theorem}

\begin{proof}
For $x\in\Sigma_2$ with $x\not=1^\infty, 0^\infty$, the projection
of each of the cylinders $C_n^-(x)$, $C_n(x)$,$C_n^+(x)$ to $\S^1$
is an interval around $\pi(x)$. Moreover we have
\begin{equation}\label{inclusion}
\pi (C_{\lfloor\kappa\log n\rfloor + 1}(x))\subset
\left(\pi(x)-\frac1{n^\kappa},\pi(x)+\frac1{n^\kappa}\right)\subset
\pi(C^*_{\lfloor\kappa\log n\rfloor}(x)).\end{equation} Applying the
left inclusion, it follows that
$$
 F^{\kappa}(\pi(x)) \subset \pi(\Fhn).
$$
Hence
$$\dim_H (F^\kappa \pi(x)) \le \dim_H \pi (\Fhn),$$
and similarly
$$\dim_H (I^\kappa \pi(x)) \ge \dim_H \pi (\Ihn).$$

We turn to the reverse inequalities.
For this we define
 $$\tau^*_n(x,y):=\inf\{l\ge 1\, :\, \sigma^lx\in
C^*_n(y)\},$$ $$\tau^-_n(x,y):=\inf\{l\ge 1\, :\, \sigma^lx\in
C^-_n(y)\}$$ and $$\tau^+_n(x,y):=\inf\{l\ge 1\, :\, \sigma^lx\in C^+_n(y)\}$$ then
\[
\tau^*_n(x,y)=\min\{\tau_n^-(x,y),\tau_n(x,y),\tau^+_n(x,y)\}
\]
and
\[
\alpha^*(x,y)=\min\{\alpha^-(x,y),\alpha(x,y),\alpha^+(x,y)\}
\]
where $\alpha^*,\alpha^-,\alpha^+$ are defined in the corresponding way. Therefore
in analogy to Lemma~\ref{lemma2}
\begin{equation}\label{contained}
\left\{\pi(y)\, :\, \alpha^*(x,y)>\frac1{\kappa}\right\}\subset
F^\kappa(\pi(x))
\end{equation}
and
\begin{equation}\label{contained'}
I^\kappa(\pi(x))\subset\left\{\pi(y)\, :\, \alpha^*(x,y)\le\frac1{\kappa}\right\}.
\end{equation}

Next we need the following lemma to prove the reverse inequalities.
\begin{lemma}\label{a+}
For any $x \in \Sigma_2^+$ and  $\nu$
an ergodic Borel probability measure
different from $\delta_{0^\infty}$ and $\delta_{1^\infty}$ we have
\[
\alpha^*(x,y)=\alpha(x,y)\qquad\nu-a.e.
\]
\end{lemma}
\begin{proof}
We will prove that $\alpha^+(x,y)\ge\alpha(x,y)$ almost everywhere.
The proof for $\alpha^-(x,y) \ge \alpha(x,y)$ a.e.\ is similar. Since
$$\alpha^*(x,y)=\min\{\alpha^-(x,y),\alpha(x,y),\alpha^+(x,y)\},$$ this
will imply the lemma.

Fix
$\epsilon>0$. Let $\1_n$ be the characteristic function of the
cylinder set consisting of $n$ 1's. Since $\nu$ is not concentrated on $1^\infty$
we can find an $n_\epsilon$ sufficiently large that
\[
\int\1_n(x)\, d\nu(x)<\epsilon \qquad (\forall n>n_\epsilon).
\]
Now let $y$ be a generic point for $\nu$.
 Then there
is an $n_0=n_0(y) > n_\epsilon$
such that
\[
\frac1m S_m\1_n(y)<\epsilon \qquad (\forall m>n_0).
\]
Let us consider the structure of $C^+_m(y)$.
\begin{eqnarray*}
    C^+_m(y) =& [y_1\cdots y_{m-1}1]      \ \ \ \ \ \        &\ \ \text{if} \ \  y=y_1\cdots
    y_{m-1}0\cdots\\
    C^+_m(y)=&[ y_1\cdots y_{k-1} 1 0 0 \cdots 0] &\ \ \text{if} \ \
         y=y_1\cdots y_{k-1}011\cdots 1 y_{m+1} \cdots.
\end{eqnarray*}
It follows that $$
        C^+_m(y) \subset C_{k-1}(y)
$$
where $k=k(y, m)$ is characterized by $y_k=0$ and $y_j=1$ ($\forall
k<j\le m\}$). Thus
\begin{equation} \label{tau+tau}
     \tau_m^+(x, y) \ge \tau_{k-1}(x, y).
\end{equation}
For a given $x$,  the more 1's at the end of $C_m(y)$ is the only way
to enlarge the difference of $x$'s
hitting times of $C^+_m(y)$ and $C_m(y)$.
Let $n>n_0$, $m > n - l - 1$ and assume that
we have a block of $n+l$ ones at the end ($l>n$). Then
(\ref{tau+tau}) becomes
\[
\tau^+_m(x,y) \ge \tau_{m-n-l-1}(x,y).
\]
The worst situation is when this block occurs very early. We are
going to estimate this first occurrence.
First we observe that
\[
\epsilon > \frac{1}{m} S_m \1_n(y) \ge \frac{l}{m} .
\]
This implies that the first occurrence of the block in question is not earlier
than
\[
m-n-l - 1 \ge  m-2l >m(1-2\epsilon).
\]
Therefore
\[
\alpha^+(x,y)=\liminf_{m\to\infty}\frac{\log\tau^+_m(x,y)}{m}\ge\liminf_{m\to\infty}\frac{\log\tau_{m-n-l-1}(x,y)}{m}\ge (1-2\epsilon)\alpha(x,y).
\]
Letting $\epsilon\to 0$ we obtain the result.
\end{proof}

We continue with the proof of the theorem. For any Borel set $A$ we have
$$\dim_H \pi A = h_{\rm top}(A)$$
since $\diam \pi(C) = 2^{-|C|}$ for any cylinder set $C$.  Thus
applying Theorem \ref{5.3} yields
$$\dim_H \pi(\Fhn) = h_{\rm top}(\Fhn) = h_{\mu_q(\kappa)\phi}.$$

Let  $t(\kappa) = q(\kappa)$ if $\frac{1}{\kappa} \ge e_{\max}$ and
$t(\kappa) = 0$ otherwise. Suppose $\e > 0$.  By continuity of
the multi-fractal spectrum we have
\[
\lim_{\e \to 0} h_{t(\kappa - \e)\phi} = h_{t(\kappa)\phi}
\]
and $$ h_{\mu_{t(\kappa-\e)\phi}}(y)  = \frac{1}{\kappa - \e} >
\frac{1}{\kappa} \quad \mu_{t(\kappa - \e)\phi}\!-\!\mbox{a.e.}\
y.$$ By Corollary \ref{fub} for $\mu_{\phi} \times
\mu_{q(\kappa -
  \e) \phi}$ for a.e.\ $(x,y)$ we have
$$\alpha^*(x,y) = \alpha(x,y) = h_{\mu_\phi}(y) > \frac{1}{\kappa}.
$$
Thus $\pi(y) \in F^{\kappa}(\pi(x))$ for $\mu_{qt\kappa - \e)\phi}$
a.e.\ $y$ and $\dim_H F_{\kappa} \ge h_ {\mu_{t(\kappa -
\e)\phi}}$. Taking the limit $\e \to 0$ shows
$$\dim_H F^{\kappa}(\pi(x)) \ge  h_ {\mu_{t(\kappa)\phi}} = \dim_H
\pi(\Fhn ).$$
This completes the proof for the set $F^{\kappa}$.

It remains to show that $\dim_H I^\kappa (\pi(x)) \le \dim_H \pi
(\Ihn).$ If $\frac{1}{\kappa} \ge e_{\max}$ then this is trivial
since $\dim_H \pi(\Ihn) = 1$. Observe that for any $\kappa$ we have
$\dim_H I^{\kappa} \pi(x) \le \frac1{\kappa}$.  To see this consider
the natural covering $(T^n \pi(x) - \frac1{n^{\kappa}},T^n \pi(x)  +
\frac{1}{n^{\kappa}})$ of $I^{\kappa}(\pi(x)).$ The $s$-covering sum
is $\sum \frac{1}{n^{\kappa s}} < \infty$ if $s > \frac{1}{\kappa}.$
Therefore, if $0 < \frac{1}{\kappa} \le h_{\mu_\phi}$, we have $\dim
I^{\kappa}(\pi(x)) \le \frac{1}{\kappa} =\dim_H(\Ihn).$ Finally   if
$ h_{\mu_{\phi}}  \le  \frac{1}{\kappa} <  e_{\max}$ the for any
H\"{o}lder function $\hat{\phi} \in H^{\alpha}(\S^1)$  let $\phi =
\hat{\phi} \circ \pi$.  We have $\phi \in H^{\alpha}(\Sigma_2)$ and
$\phi(x_1, \dots,x_n 0 1^{\infty}) = \phi(x_1,\dots,x_n,1
0^{\infty})$ thus by the Gibbs property we have
$$\lim_{n \to \infty} \frac{\log \mu_{\phi}(C^{\pm}_n(x))}{\log
  \mu_{\phi} (C_n(x))} = 1.$$
Hence  $h^*_{\mu_{\phi}}(y) = h_{\mu_{\phi}}(y)$ for all $y \in \Sigma_2$.

Consider the set
\begin{eqnarray*}
I^{\kappa}(\pi(x)) \backslash \pi(\Ihn) & = & \{y: \pi(y) \in
I^{\kappa}(\pi(x)), y \in \Fhn  \}\\
& \subset & \{y: \alpha^*(x,y) < \frac{1}{\kappa}, \alpha(x,y) \ge \frac{1}{\kappa}  \}.
\end{eqnarray*}
By Theorem \ref{thm5.2} for $\mu_{\phi}$-a.e.\ $x$ we have that the
last set is contained in
$$\{y: \alpha^*(x,y) < \frac{1}{\kappa}, h_{\mu_{\phi}}(y)
\ge \frac{1}{\kappa}\}.$$
Thus Lemma \ref{lemma1add}  implies that for any $\e > 0$ there are at most $C(\e)
\cdot 2^{}E(\frac{1}{\kappa})n$
cylinders of length $n$ needed to cover $\{y: \alpha^*(x,y) < \frac{1}{\kappa}, h_{\mu_{\phi}}(y)
\ge \frac{1}{\kappa} + \e\}$ .  Hence
$$\dim_H(I^{\kappa}(\pi(x)) \backslash \pi(\Ihn) \le
E(\frac{1}{\kappa}) = \dim_H \pi(\Fhn).$$
\end{proof}

\begin{corollary} For $\mu_{\phi}$ a.e.\ $x$ we have
$F^{\kappa}(\pi(x)) = \emptyset$ if $\frac{1}{\kappa} > e_+.$
\end{corollary}

In Theorem~\ref{var} and Corollary~\ref{fub} we can ignore the delta
measure on fixed points since they have zero entropy and therefore
do not give any contribution. This transfer procedure allows us to
conclude the following Theorems and Corollaries from the analogous
results of the section~\ref{sec8}. These results contain more
information than those stated in the introduction, thus we
reformulate them. We set $\nu_\phi = \mu_\phi \circ \pi^{-1}$.

\begin{theorem} {\rm (Theorem \ref{thm1.1})} \label{nuc}
$\kappa_{\phi,\psi} = \frac{1}{- \int_{\S^1} \phi \, d\nu_{\psi}}
= \frac{1}{h_{\nu_{\phi}}(y)} = - \frac{1}{\frac{d}{dt}P(\phi + t \psi)|_{t=0}}.$
\end{theorem}

\begin{lemma}\label{lemma3'}
For $\nu_\phi$ a.e.\ $s$ we have
$$
\sup \{E(t): \frac1t \le \kappa\} \ge \dim_H \Fn \ge  \sup \{E(t): \frac1t < \kappa\}.$$
For $\nu_\phi$ a.e.\ $s$ and $\kappa < 1/h_{\nu_\phi}$ we have
$$
\sup \{ E(t): \frac1t \ge \kappa\} \ge \dim_H \In \ge  \sup \{ E(t): \frac1t > \kappa\}.
$$
\end{lemma}

\begin{corollary}\label{theoremMFA'}
For $\nu_{\phi}$ a.e.\ $s$
$$
\sup_{-P'(q) \ge \frac{1}{\kappa}}\left[ P(q\phi) - P'(q\phi)q\right ]  \ge
dim_H \Fn
\ge  \sup_{-P'(q) > \frac{1}{\kappa}}\left[ P(q\phi) -
  P'(q\phi)q\right ].
$$
For $\nu_\phi$ a.e.\ $s$ and for $\kappa < 1/h_{\nu_\phi}$
$$
\sup_{-P'(q) \le \frac{1}{\kappa}}\left[ P(q\phi) - P'(q\phi)q\right ]  \ge
dim_H \In
\ge \sup_{-P'(q) < \frac{1}{\kappa}}\left[ P(q\phi) - P'(q\phi)q\right ].
$$
\end{corollary}

\begin{corollary} {\rm (Theorems \ref{thm1.3} and \ref{thm1.4})}
For a typical potential $\phi$ and $\nu_\phi$ a.e.\ $s$ we have
\begin{eqnarray*}
&\dim_H \Fn = \dim_H(\S^1) = h_{\rm top}(\S^1) = 1  &\hbox{for } 1/\kappa \le - \int \phi \, dLeb.,\\
&\dim_H \In =  \dim_H(\S^1) = h_{\rm top}(\S^1) = 1  & \hbox{for }
1/\kappa \ge - \int \phi \, dLeb..
\end{eqnarray*}
Let $q_{\kappa}$ be the number such that $P'(q_{\kappa} \phi) = - \frac1{\kappa}.$ Then
\begin{eqnarray*}
&\dim_H \Fn = E\left (\frac{1}{\kappa} \right ) = P\left (q_{\kappa}
\phi \right ) + \frac{1}{\kappa} q_{\kappa}
\quad &\hbox{for } 1/\kappa > - \int \phi \, dLeb.,\\
&\dim_H \In = E\left (\frac{1}{\kappa} \right ) = P\left (q_{\kappa}
\phi \right ) + \frac{1}{\kappa} q_{\kappa} \quad &\hbox{for }
h_{\nu_\phi} \le 1/\kappa < - \int \phi \, dLeb.,\\
&\dim_H \In = \frac{1}{\kappa} \qquad \qquad \qquad\qquad\quad \quad
&\hbox{for } 1/\kappa < h_{\nu_\phi}.
\end{eqnarray*}
\end{corollary}

{\bf Remark:} 1)  If $\kappa >1$ then $\sum l_n < \infty$ and
we can not cover Lebesgue almost all points infinitely often
no matter which orbit we consider.
Thus it is likely that the dimension of $\In$ is less than 1.
In the degenerate case this is clear.
To see this in the nondegenerate case
note that since the graph of the entropy spectrum is below the diagonal
we have $1 = h_{\rm top} = E(e_{\max}) < e_{\max}$.
Therefore the maximum dimension (i.e.\ 1) is attained for $\kappa < 1$.

2)  For a non typical potential we have possibly discontinuities
of the function $E(t)$ at $e^\pm$.  At these points the upper and lower
estimates of Corollary~\ref{theoremMFA'} do not coincide.  This indicates that
the question about infinite versus finite covering can not be completely answered
in terms of the exponent $\kappa$.   At this point the answer might depend on
a constant $c$ where $l_n = \frac{c}{n^\nu}$.  This is in particular the case
for the i.i.d.\ case mentioned in the introduction.  The dynamical analog is Lebesgue
measure whose entropy spectrum is degenerate.  Therefore we can not get any
information about the sequence $\frac{c}{n}$ which resembles the i.i.d.\ case.

\begin{theorem} {\rm (Theorem \ref{thm1.2})}
For $\nu_\phi$ a.e.\ $s$ we have
$$
\Fn = \emptyset   \hbox{ for }  \kappa < \frac{1}{- \inf_{\mu \ \rm{ ergodic}} \int \phi \, d\mu}= \kappa^F_{\phi,\S^1}.
$$
\end{theorem}

These results are summarized in Figure 2.
\begin{figure}[t]
\psfrag*{a}{\footnotesize{$\dim_H \Fn$}}
\psfrag*{b}{\footnotesize{$\dim_H \In$}}
\psfrag*{c}{\footnotesize{$e^-$}}
\psfrag*{d}{\footnotesize{$e_{\max}$}}
\psfrag*{e}{\footnotesize{$e^+$}}
\psfrag*{f}{\footnotesize{$1$}}
\psfrag*{g}{\footnotesize{$1/\kappa$}}
\psfrag*{h}{\footnotesize{$e^-=e_{\max}=e^+$}}
\centerline{\psfig{file=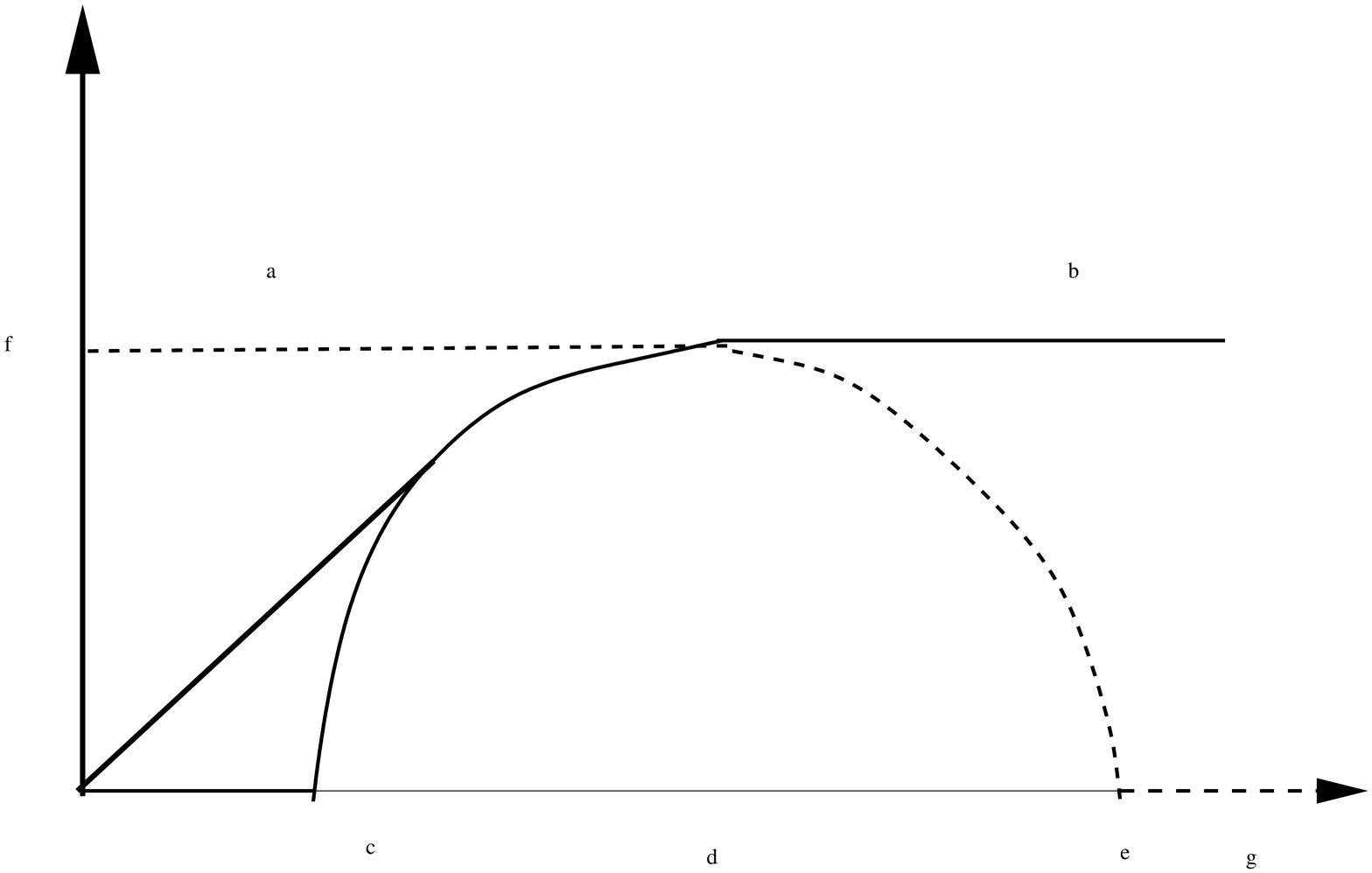,height=45mm}}
\centerline{\psfig{file=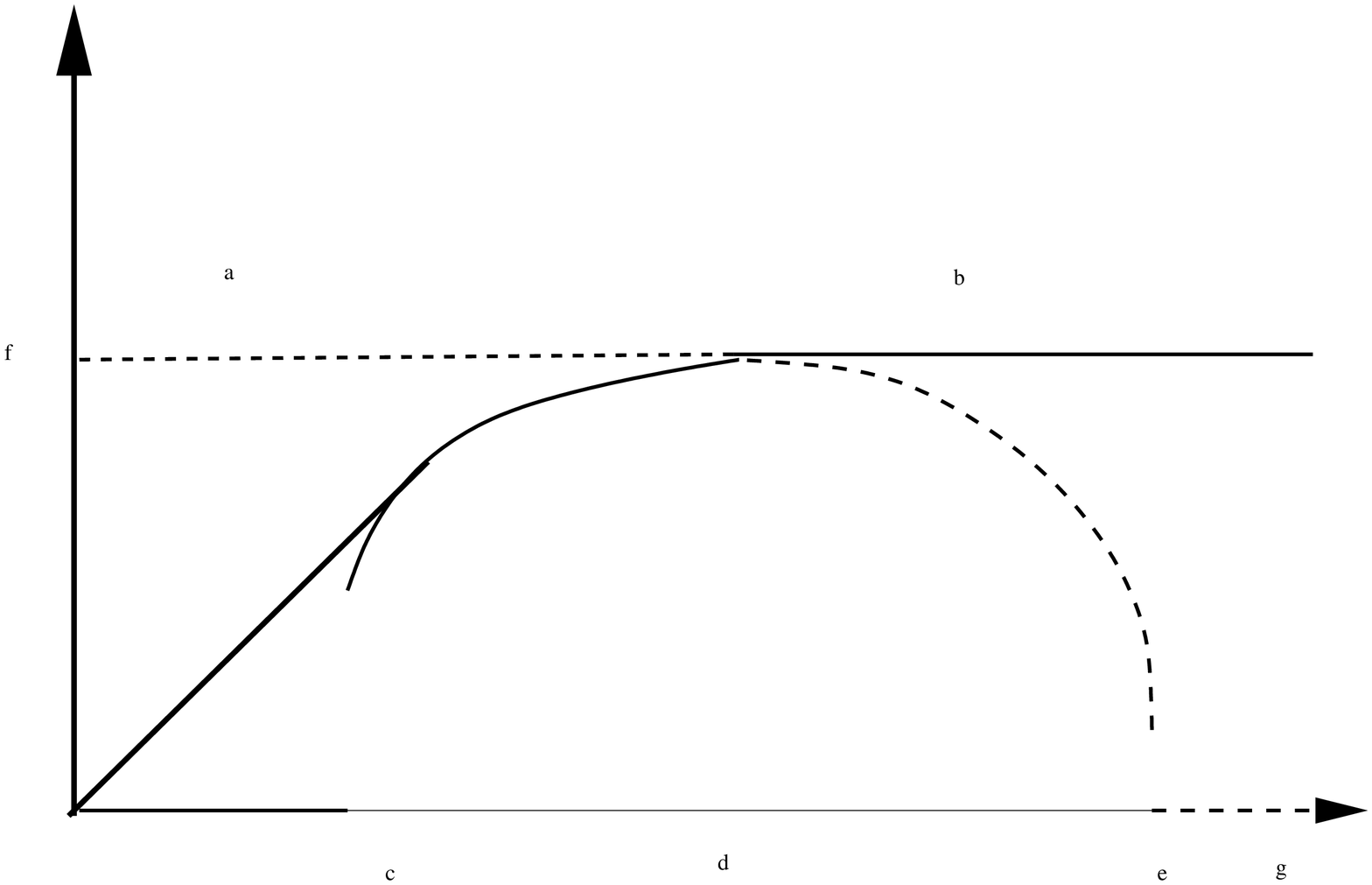,height=45mm} \quad
\psfig{file=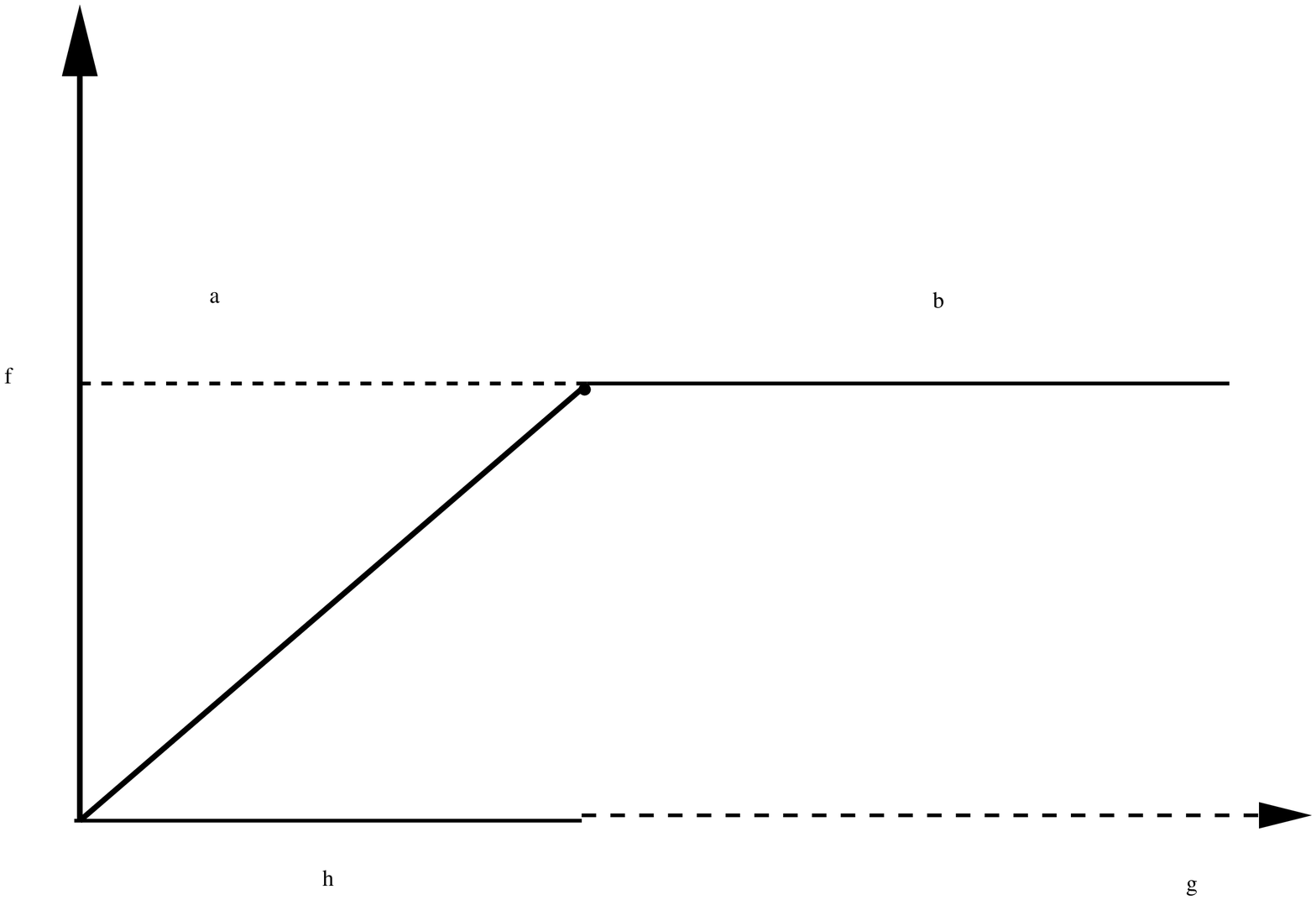,height=45mm}}

\caption{The dimension graphs in the typical, nontypical and
  degenerate cases.
The graph of $\dim_H \Fn$ is dotted and the graph of
$\dim_H \In$ is solid.}\label{cov2}
\end{figure}

{\bf Remark:}  The result of Corollary \ref{structure} can also be
transferred to the circle.  The interpretation of this result is as
follows.  The distribution of a typical orbit up to time $L$ is
completely determined by the entropy spectrum of the measure.

\medskip
{\bf Acknowledgment} The authors thank L.--M.\ Liao and Q.--L.\ Li for
their useful remarks.

\end{document}